\documentclass[12pt]{amsart}
\usepackage{amssymb}
\usepackage{graphicx}
\usepackage{epstopdf}
\usepackage{color}
\usepackage{overpic}
\input{epsf}

\theoremstyle{plain}
\newtheorem{theorem}{Theorem}[section]
\theoremstyle{definition}
\newtheorem{definition}[theorem]{Definition}
\newtheorem{proposition}[theorem]{Proposition}
\newtheorem{remark}[theorem]{Remark}
\newtheorem{lemma}[theorem]{Lemma}
\newtheorem{corollary}[theorem]{Corollary}
\newtheorem{example}[theorem]{Example}

\newcommand{\num}{\renewcommand{\labelenumi}{(\roman{enumi})}
  \renewcommand{\theenumi}{(\roman{enumi})}}
\def\dash{\discretionary{-}{}{-}\penalty1000\hskip0pt}

\newcommand{\R}{\mathbb{R}}

\DeclareMathOperator{\sk}{\mathit{sk}}

\DeclareMathOperator{\hf}{H\Phi}

\makeatletter\def\@captionfont{\normalfont\footnotesize}\makeatother

\title{Equilibrium stressability of multidimensional frameworks}

\author[KMPSSS]{Oleg Karpenkov, Christian M\"uller, Gaiane Panina,
Brigitte Servatius, Herman Servatius, Dirk Siersma}

\address{Oleg Karpenkov\\University of Liverpool
}
\email{karpenk@liv.ac.uk}

\address{Christian M\"uller\\
TU Wien
}
\email{cmueller@geometrie.tuwien.ac.at}

\address{Gaiane Panina\\
PDMI RAS, St. Petersburg State University
}
\email{gaiane-panina@rambler.ru}

\address{Dirk Siersma\\
University of Utrecht}
\email{d.siersma@uu.nl}

\address{Brigitte Servatius\\
Worcester Polytechnic Institute}
\email{bservat@wpi.edu}

\address{Herman Servatius\\
Worcester Polytechnic Institute}
\email{hservat@wpi.edu}

\thanks{%
The collaborative research on this article was supported by the
``Research-In-Groups'' program of ICMS Edinburgh, UK.
O.~Karpenkov is partially supported by EPSRC grant EP/N014499/1 (LCMH) and
C.~M{\"u}ller by the Austrian Science Fund (FWF) through project P~29981.
}


\begin{document}

\begin{abstract}
  We prove an equilibrium stressability criterium for trivalent
  multidimensional tensegrities. The criterium appears in different
  languages: (1) in terms of stress monodromies, (2) in terms of
  surgeries, (3) in terms of exact discrete 1-forms, and (4) in Cayley
  algebra terms.
% From Giane Tuesday
%
% Old version
%In this paper we develop geometric conditions for existence of multidimensional tensegrities on $(3,4)$-valent simple CW-complexes
%of dimension $d{-}1$ in $\R^d$.
%These conditions are expressed in terms of ``meet'' and ``join'' operations in the Cayley algebra.
%The proposed approach gives a geometric approach to the factorisation problem for tensegrities on $(3,4)$-valent simple CW-complexes.

%To be published in

%1. International Journal of Solids and Structures

%2. European J

%3. Discrete and Computational Geometry

%4. Beitrage Geometry

\end{abstract}

\maketitle
\tableofcontents

\section{Introduction}\label{introduction}
In the previous century, Fuller~\cite{marks1960dymaxion} coined the term {\em tensegrity}, a combination of `tension' and `integrity',
to describe networks of rods and cables, such as those created by artist Kenneth Snelson, in which the tension of the cables and the compression
in the rods combine to yield structural integrity to the whole.
More generally, the word tensegrity is used to describe a variety of
practical and abstract structures, e.g.\ bicycle tires and tents, whose rigidity follows from the balance of {\em tension} and {\em compression},
in the mathematical literature the {\em stress}, on the members.

   Practically, structures exhibiting tensegrity may be generated and analyzed using conventional techniques of structural
   engineering~\cite{pellegrino2003}.  Theoretically, tensegrity is often considered
   as part of the study of geometric constraint systems,~\cite{Connelly2013,sitharam2018handbook}.   The classical tensegrity model consists of a
    set of vertices $V$, and
   two graphs, $(V,S)$ and $(V,C)$, the graph of {\em struts} and {\em cables}, and a placement function $\mathbf{p}:V \rightarrow \mathbb{R}^d$.
   One looks for a motion of the placed vertices such that the distances between pairs of vertices connected by struts do not
   fall below their initial values, and such that the distances between pairs of vertices connected by cables do not
   expand beyond their initial value.  If no such motion exists, apart from the rigid motions of the space itself, the system is said to be {\em rigid}.
   A {\em stress} is a function $s:S \cup C \rightarrow \mathbb{R}$,
   with $s(t) \geq 0$ and $s(c) \leq 0$ for all cables, $c \in C$,  and
   struts, $t \in S$.  A stress is an {\em equilibrium stress} if for each vertex $v$
      $$ \sum_{(v,w) \in C \cup S} s((v,w)) (\mathbf{p}(v) - \mathbf{p}(w)) = \mathbf{0}. $$

   There are two avenues in which the existence of a proper, i.e.\ nowhere zero,  equilibrium stress may allow one to establish structural integrity.  A result of
   Roth and Whitely~\cite{rothwhiteley1981} states that if a tensegrity has a proper equilibrium stress, and if the placement for
   $(V,C \cup S)$ is statically rigid as a
   bar and joint framework, then that tensegrity must be first order rigid and hence
   rigid. More delicately, if the proper equilibrium stress passes the {\em second-order stress test} of Connelley and Whiteley~\cite{Connelly1990SecondOrderRA},
   then the tensegrity structure is second order rigid, hence rigid.

   Another important aspect of proper equilibrium stresses is the connection between the existence of an equilibrium stress
and the lifting of embedded graphs into higher dimensions, as provided by
the theory of Maxwell-Cremona.
It was understood by Lee, Ryshkov, Rybnikov, and some others that the connection between equilibrium stresses and lifts extends in certain cases to
arbitrary CW-complexes $M$ realized in any dimension $d$, also not necessarily embedded.
Although the subject can be traced earlier (see \cite{Lee96,whiteley1989}), the first systematic study of multidimensional stresses,
liftings and reciprocal diagrams was undertaken by Rybnikov in
\cite{Rybnikov991,Rybnikov99}.

In the present paper we introduce $d$-frameworks and their equilibrium
stresses, i.e.\ self-stresses, in a slightly broader way than it was done
by Rybnikov (Section~\ref{sectionmaindefs}).
We introduce face paths and stress transition along such face paths in
$d$-frameworks in Section~\ref{sectionstressability}.
We also give (Section~\ref{sectionselfstressabilitycriterium}) necessary
and sufficient conditions for the equilibrium stressability of trivalent
frameworks. This result appears as a generalization of equilibrium
stressability criteria for classical
tensegrities~\cite{karpenkov+2019-arXiv}.
The conditions are equivalently expressed in terms of exact discrete
multiplicative $1$-forms, Cayley algebra, or, in terms of some surgeries
introduced in Section~\ref{sectionselfstressabilitycriterium}.

In Section~\ref{sectionstressabilityII},
we explain how Rybnikov's frameworks (R-frameworks, for short) arise in the proposed context. We give an equilibrium stressability criterion
and derive some examples that demonstrate similarities and differences between planar and multidimensional tensegrities.

\section{Main definitions and constructions}\label{sectionmaindefs}
Our model for tensegrity in this paper will be based on the following structure.
Let $D>d\ge 1$ be two integers. In the sequel, the term \emph{plane} means
an affine subspace in $\R^D$.

\begin{definition}\label{dframeworkdefinition}
A {\em $d$-framework} $\mathcal{F}=(E,F,I,\mathbf{n})$ consists of
$E$, a collection of $(d{-}1)$-dimensional planes in $\R^D$;
$F$, a collection of $d$-dimensional planes in $\R^D$;
a subset $I \subset \{(p,q) \in (E \times F) \mid p\subset q\}$;
a function $\mathbf{n}$ assigning to each pair
$(e,f)$ with $e\in E$ and $(e,f)\in I$, a unit vector
$\mathbf{n}(e,f)$  which is contained in $f$ and which is normal to $e$.
We call planes from $F$ \emph{faces} and planes  from $E$
\emph{edges}.
The set $I$ is called the {\it set of incidences}.
\end{definition}

A $d$-framework is called \emph{generic}
if, for every $e\in E$, all the planes $f$ with $(e,f)\in I$ are distinct.

%N(E,F,I)$:
%for every pair $(e,f)\in I$ where $e\in E$ and $f\in F$,
%we pick a normal to the $(d-1)$-plane $e$ in the $d$-plane.
%We denote it by $\mathbf{n}(e,f)$. The collections of such normals is  $N(E,F,I)$.
%\end{itemize}
%
%N:  A function which assigns to each pair $(e,f)$, with $e \in E$ and $(e,f) \in F$
%a vector $N(e,f)$ such that $N(e,f)$ is normal to $e$ and
%lies in $f$

%(or {\it $d$-framework}, for simplicity).
%\end{definition}

%\begin{remark}
%Note that inclusions of planes $p \subset q$ for pairs other those specified in $I$ are allowed.
%\end{remark}

%\begin{remark}
%The definition was inspired by constructions of tents where the surfaces of
%tent material are stressed. It addition, the term tenting has a two-dimensional
%flavor. Indeed tenting $d$-frameworks has only planes of three different dimensions,
%generalizing vertices, edges, and faces.
%\end{remark}

Let $\mathcal{F}=(E,F,I,\mathbf{n})$ be a $d$-framework.
A \emph{stress} $s$ on $\mathcal{F}$ is any function $s:F\to \R$.
A framework $\mathcal{F}$ together with a stress $s$ is said to be in {\em
equilibrium} if for every $e\in E$ we have
\begin{equation}\label{self-stress-d-framework}
\sum_{(e,f)\in I} s(f)\mathbf{n}(e,f)=0.
\end{equation}
Such a stress is called an {\em equilibrium stress},  a {\em  self-stress} or sometimes
a {\em prestress} for $\mathcal{F}$.
\begin{definition}\label{tensegritydefinition}
A $d$-framework is said to be \emph{self-stressable} or a
{\it tensegrity} if there exists a non-zero self-stress on it.
\end{definition}

\begin{example}
The simplest non-trivial example here is the classical case of graphs in the plane ($D=2,d=1$).
\end{example}

We say that a $d$-framework is \emph{trivalent} if each element of $E$ is
incident (i.e., contained in a pair in $I$) to precisely 3 elements of
$F$.

Surface based tensegrities of this type are models for minimal surfaces (or, more generally, harmonic surfaces)
which meet at edges, such as, soap bubbles or tents.  In this model, in
the $D = 3$ case, the surfaces are flat, and we can think of them as rigid
plates, each having an expansion or contraction coefficient, say caused by
heat or cooling, for which the equilibrium condition indicates that the
forces cancel on the edges, so that the framework does not deform.

\begin{example}\label{K5stresssEx}
  Let $d = 2$ and  $D = 3$. Consider a $2$-framework whose edges, $E$,
  and faces, $F$, correspond to the edges and triangles of the graph $K_5$, embedded
  in $\mathbb{R}^3$.   If four of the vertices, $\{1,2,3,4\}$, of $K_5$ are placed as vertices of a regular tetrahedron
and the fifth one as their centroid, then the resulting $2$-framework is generic in our sense.
For each edge-face pair choose the normal to be a unit vector pointing into the interior of the
face-triangle.
Note that the chosen normals sum to the zero vector around the lines corresponding to edges $\{ i,5\}$,
so choosing equal stress on these interior triangles leaves those edges equilibrated.  Then, it is easy to see
that choosing stresses on the exterior and interior triangles in the ratio
$-\sqrt{6}/4$ yields an equilibrium stress.
%Now consider the edge $e= \{ 1,2\}$. It is contained in the faces $ f_1 = \{ 1,2, 3\}$, $ f_2 =\{ 1,2,4\}$ and $f_3 =\{ 1,2, 5\}$.
%Here $\alpha (\mathbf{n}(e,f_1 ) + \mathbf{n}(e, f_2))= \mathbf{n}(e,f_3)$. Assigning stress~1 to all faces containing vertex ~5
%and $- \alpha$ to all the faces not containing vertex $5$ yields a non-zero self-stress.

In this example one may imagine the interior expanding triangles exerting an outward force balanced by the contracting
``skin'' of the exterior  triangles.
\end{example}

\begin{example}\label{K5nonstressEx}
Again, let $d = 2$ and  $D = 3$.  We may create a different $2$-framework based on the graph
$K_5$ in $\mathbb{R}^3$
by keeping $E$ as before, and associating the faces $F$ to the $K_4$ subgraphs of $K_5$, with the usual incidence relation.
Since any two $K_4$'s intersect in $3$ edges,
their face planes must be identical, and the $5$ vertices of the embedded $K_5$ must be coplanar.

Since the normal vectors all lie in the plane of the $K_5$, it is no loss of generality to assume that the
vertices lie on a regular pentagon, and it is quickly checked that only the zero stress
satisfies Equation~\eqref{self-stress-d-framework}.
\end{example}

\begin{example}\label{CubeEx} Let $d = 2$ and  $D = 3$.
  Consider the vertices of a regular cube and set $E$ to be the set of all
  lines joining a pair of non-antipodal vertices (see
  Figure~\ref{tense2019cobeexample01}). The face planes $F$ consist of all six
  planes containing the faces of the cube, together with
  the six planes containing antipodal pairs of cube edges,
  as well as the eight planes of the dual tetrahedra.
  \begin{figure}[htb]
    \centering
    \includegraphics[scale=.9]{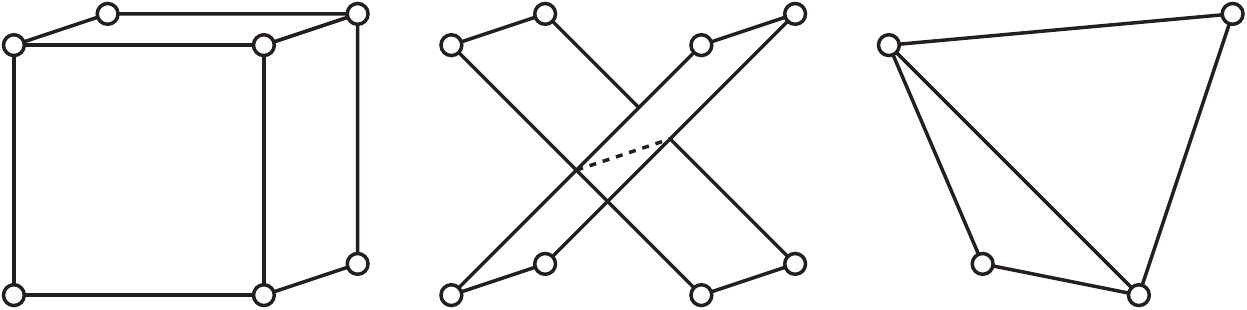}
    \caption{A $3$-framework based on the cube with three types of faces.}
    \label{tense2019cobeexample01}
  \end{figure}
  Let incidences be induced by containment. Since each plane contains a
  polygon of edges supported by incident lines of the structure, we may
  take the normals to be inwardly pointing
  unit vectors. It is easy to check that the self-stresses on the three
  types of faces are in the ratio
  $1:-\sqrt{2}:\sqrt{3}/4$.
\end{example}

\begin{example}\label{OctahedronEx}
This example has two versions, both with $d = 2$ and $D = 3$.
Consider the vertices of an octahedron, regularly embedded in $\mathbb{R}^3$.
\begin{figure}[htb]
  \centering
  \includegraphics[scale=1.1]{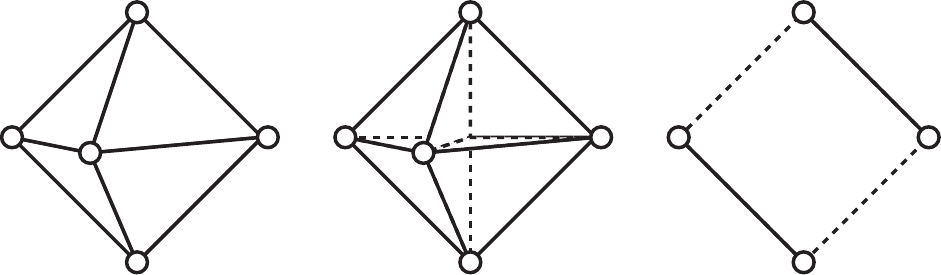}
  \caption{Three face types: triangles, squares, or faces with just two
    edges.}
  %\label{tense2019moebius01fig}
\end{figure}
The set of 12 edge lines $E$  lie along the edges of the octahedron, and
the set $F$ of 11 face planes will consist of those eight supporting triangles of
the octahedron, together with the three planes which pass through four coplanar vertices.
Let the incidences be all those induced by containment and, as before, let all normals be
chosen inwardly pointing with respect to the triangle or square to which they belong.
Then this is easily computed to be stressable, and hence a tensegrity. In fact, the self-stress
is unique, since each each line is incident to three distinct planes.

As an alternative, we can take each of the three planes containing four vertices to have
multiplicity $2$, with each one incident to a different
pair of opposite lines, and with the same choice of normals.
This structure consisting of 12 lines and 14 planes is also a tensegrity.
\end{example}

%\begin{remark}
%Let $F$ have some repeating planes.
%For the stresses and self-stresses we allow to have different values for different copies of the same planes.
%\end{remark}

\section{Self-stressability of frameworks.}\label{sectionstressability}
In this section we study self-stressability of trivalent $d$-frameworks in
$\R^{d+1}$, so starting from now on, we assume that $D=d+1$.

\subsection{Self-stressability of face-paths and face-cycles}\label{stress-s2}
Let us start with the following general definition.

\subsubsection{Face-path and face-cycle}% $d$-frameworks}
 Let $E = (e_1, \ldots, e_k)$ be
a sequence of distinct $(d{-}1)$-dimensional planes in $\R^D$;
$F = (f_{0,1}, f_{1,2}, \ldots, f_{k,k+1})$ be a sequence of $d$-dimensional
planes in $\R^D$;
$\hat F = (\hat f_{1}, \ldots, \hat f_k)$ be a sequence of $d$-dimensional
planes in $\R^D$.
Then the collection $(E,F,\hat F)$ is said to be a \emph{face-path} if for
$i=1,\ldots,k$ we have
$$
e_i\subset f_{i,i+1}, \quad
e_i\subset f_{i-1,i}, \quad \hbox{and} \quad
e_i\subset \hat f_{i}.
$$
For example, removal of two antipodal triangles of an octahedron leaves
a face-cycle with six faces.
%see Figure~\ref{fig:facepath} for a three-dimensional illustration.
%% BRIGITTE DOES NOT LIKE REFERRING TO THIS FIGURE HERE

Denote the set of pairs defined by  these inclusions by $I$.
For a particular choice of normals $\mathbf{n}$ we obtain a $d$-framework
$\mathcal{F} = (E,F\cup\hat F, I, \mathbf{n})$ which is called a
\emph{face-path $d$-framework}, see Figure~\ref{tense2019moebius01a}.
It
is called a \emph{face-cycle $d$-framework} if $f_{0,1}=f_{k,k+1}$. In
this case it is denoted by $C(E,F,\hat F,\mathbf{n})$,
see~Figure~\ref{tense2019moebius01c}.
\begin{figure}[t]
  \centering
  \includegraphics{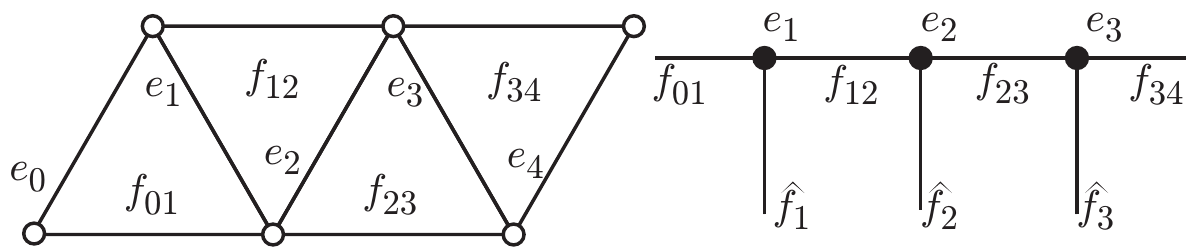}
  \caption{A face path contained in the framework described in
  Example~\ref{K5stresssEx} with $f_{0,1}= \{1,2,4\}$, $f_{1,2}=
  \{2,3,4\}$, $f_{2,3}= \{3, 4, 5\}$, $f_{3,4}= \{1,3,5\}$.}
  \label{tense2019moebius01a}
\end{figure}

\begin{figure}[htb]
  \centering
  \includegraphics{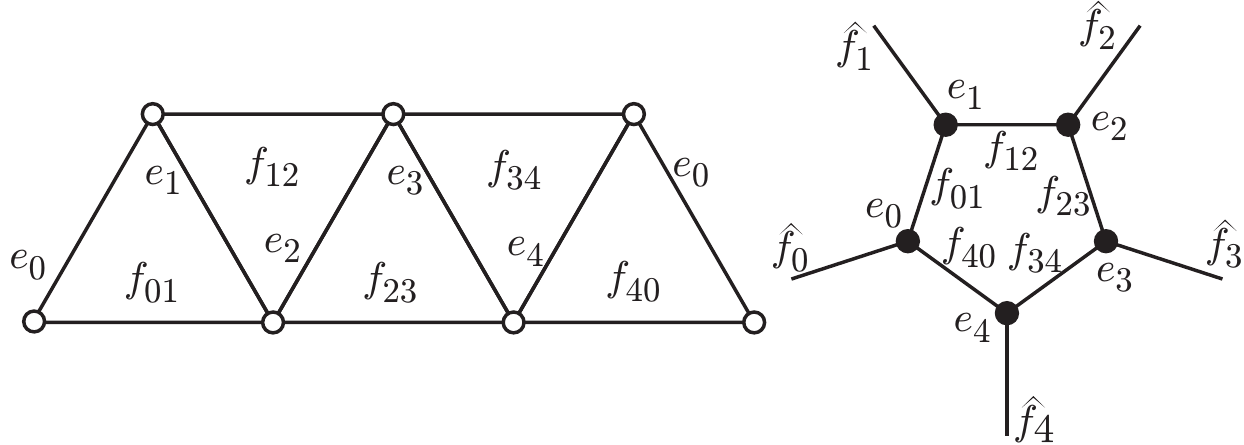}
  \caption{A face path contained in the framework described in
  Example~\ref{K5stresssEx}, with labeling as before and
  $f_{4,0}= \{1,2,5\}$.}
  \label{tense2019moebius01c}
\end{figure}

%\end{definition}
%\begin{definition}
%Let $d\ge 2$ be an integer, set $D=d+1$. Consider
%\begin{itemize}
%\item $E{=}(e_1,\ldots,e_k)$ be a sequence of distinct $(d{-}1)$-dimensional planes in $\R^D$;
%
%\item $F{=}(f_{0,1},f_{1,2},\ldots, f_{k,k+1})$ be a sequence of $d$-dimensional planes in $\R^D$;
%
%\item $\hat F{=}(\hat f_{1},\ldots, \hat f_k)$ be a sequence of $d$-dimensional planes in $\R^D$;
%
%\begin{itemize}
%\item The collection $C=(E,F,\hat F)$ is said to be a {\it face-path}
%if for $i=1,\ldots,k$ we have
%$$
%e_i\subset f_{i,i+1}, \quad
%e_i\subset f_{i-1,i}, \quad \hbox{and} \quad
%e_i\subset \hat f_{i}.
%$$
%Denote the set of pairs in these inclusions by $I$.
%\item
%For an arbitrary choice of normals $\mathbf{n}$
%we say that the $d$-framework $T(\emptyset,E,F\cup\hat F, I, \mathbf{n})$
%is a {\it face-path $d$-framework}.
%
%\item In case if $f_{0,1}=f_{k,k+1}$ are identified (as elements)
%we say that the $d$-framework $T(\emptyset,E,F\cup\hat F, I, \mathbf{n})$
%is a {\it face-cycle $d$-framework}.
%We denote it by
%$$
%C(E,F,\hat F,\mathbf{n}).
%$$
%\end{itemize}
%\end{itemize}
%\end{definition}

%\begin{figure}[htb]
%\centering
%\includegraphics{tense2019moebius01}
%\caption{A face cycle of length 5, DELETE ARROWS \label{tense2019moebius01}}
%\end{figure}

\subsubsection{Self-stressability of face-path $d$-frameworks}

\begin{proposition}\label{path-self-stress-existence}
  Any generic face-path $d$-framework which contains no face-cycle has a
  one-dimen\-sional space of self-stresses.  All the stresses for all the
  planes \ of $F$ and $\hat F$ are either simultaneously zero, or
  simultaneously non-zero.
\end{proposition}

\begin{proof}
  Setting  $s(f_{0,1})=1$ we inductively define all stresses for all
  other planes  using Equation~\eqref{self-stress-d-framework}. Therefore a
  non-zero self-stress exists.  By construction the obtained stress is
  non-zero at all planes of $F$ and $\hat F$.

  Once we know any of the stresses at one of the planes of $F$ and $\hat
  F$, we reconstruct the remaining stresses uniquely using
  Equation~\eqref{self-stress-d-framework}. Hence the space of stresses is
  at most one-dimensional.  Therefore, all self-stresses are proportional
  to a self-stress that is non-zero at all planes of $F$ and $\hat F$.
\end{proof}

\subsubsection{Edge-orientation transition}
Suppose we have a face path whose edges are $e_1, e_2, \ldots$, and we
are given an orientation of $e_i$ by declaring a frame spanning $e_i$ as
positive. We pass over this positive orientation on $e_i$ to a positive
orientation on $e_{i+1}$ by requiring that
the given frame of $e_i$ together with $\mathbf{n}(e_i,f_{i,i+1})$ and
a new (chosen to be positive) frame of $e_{i+1}$ together with
$\mathbf{n}(e_{i+1},f_{i,i+1})$
differ by an orientation {reversing} automorphism on $f_{i,i+1}$.
We call this the {\em edge-orientation transition}.

A face-cycle $d$-framework $C(E,F, \hat F,\mathbf{n})$ is said to be
\emph{edge-orientable} if the edge-orientation transition around the cycle
returns to the starting edge in its initial orientation.
%the positive basis of $e_n$ together with $\mathbf{n}(e_k,f_{k,k+1})$ and
%the positive basis of $e_{1}$ with  $\mathbf{n}(e_{0},f_{0,1})$
%should give opposite orientations of $f_{k,k+1}=f_{0,1}$.
%Otherwise we say that the face-cycle $d$-framework is not {\it edge-orientable}.

%\begin{remark} %NOT LABELLED OR REFERRED TO (HS)
    Non-orientable face-cycles are a usual phenomenon in frameworks. Indeed, we see them even in small examples
like Example~\ref{K5stresssEx}.
 Figure~\ref{tense2019moebius01fig} depicts such a face-cycle. Here the
  first and the last edges coincide, but are oppositely oriented.
%\end{remark}
\begin{figure}[htb]
  \centering
  \includegraphics{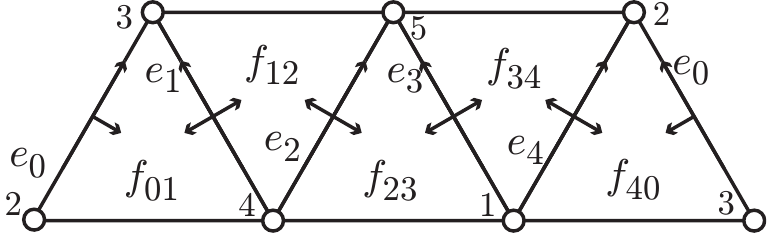}
  \caption{A  face-cycle contained in $K_5$  which is not edge-orientable.  For this particular example faces correspond to triangles, and we choose all the
  normals to  point inward.}
  \label{tense2019moebius01fig}
\end{figure}

Note that the edge-orientability of a cycle depends neither on the choice
of the first element $e_1\in E$, nor the choice of direction in the cycle.
We observe the following proposition.

\begin{proposition}\label{Changes}
A face-cycle $d$-framework $C=C(E,F,\hat F,\mathbf{n})$ has the following
  properties.
\begin{enumerate}\num
  \item\label{or:i} Reversing simultaneously all the normals at a single
    $e_i\in E$ (namely $\mathbf{n}(e_i,f_{i-1,i})$,
    $\mathbf{n}(e_i,f_{i,i+1})$, and $\mathbf{n}(e_i,\hat f_{i})$) does
    not change the self-stressability or orientability of $C$.
  \item\label{or:ii} Reversing simultaneously the normals at $f_{i,i+1}\in
    F$ (namely $\mathbf{n}(e_i,f_{i,i+1})$ and
    $\mathbf{n}(e_{i+1},f_{i,i+1})$) does not change
    the self-stressability or orientability of $C$.
  \item\label{or:iii} Reversing the normal $\mathbf{n}(e_i,\hat f_{i})$ does
    not change the self-stressability or orientability of $C$.
\end{enumerate}
\end{proposition}

\begin{proof}
  In all the items we change altogether an even number of normals for all
  the faces in $F$.  Therefore, orientability is preserved.

  The change in \ref{or:i} does not change the equations of
  self-stressability, so it preserves self-stressability.  For \ref{or:ii}
  and \ref{or:iii} the change of the signs of stresses $s(f_{i,i+1})$ and
  $s(\hat f_i)$ respectively delivers the equivalence of the conditions of
  self-stressability.
\end{proof}

\subsubsection{Self-stressability of face-cycles of length 3}
Let us consider a  trivalent cycle  $C(E,F,\hat
F,\mathbf{n})$ of length $3$ with $$F=\{f_1,f_2,f_3\}, \hat
F=\{\hat f_1, \hat f_2, \hat f_3\}, E=\{e_1,e_2,e_3\} $$ (for a schematic sketch see
Figure~\ref{fig:triangle}).
Create a new  plane $\hat f'_3$ by the following
 Cayley algebra algorithm (see
Figure~\ref{fig:harmonic}):

(i)~$g_1 = (e_1 \vee e_2) \wedge \hat f_3$,

(ii)~$g_2 = f_{2, 3} \wedge \hat f_1$,

(iii)~$g_3 = (g_1 \vee g_2) \wedge f_{3, 1}$,

(iv)~$g_4 = (g_2 \vee e_1) \wedge (g_3 \vee e_2)$,

(v)~$\hat f'_3 = e_3 \vee g_4$.

In this notation, the stressability conditions are given by the following.
\begin{proposition}\label{three-cycle}
  A face-cycle $d$-framework
  $$
  C\big((e_1,e_2,e_3), (f_{1,2},f_{2,3},f_{3,1}),
  (\hat f_{1}, \hat f_2, \hat f_3),\mathbf n\big)
  $$
  is self-stressable if and only if
    $$\dim(\hat f_1\cap \hat f_2\cap \hat  f_3) = d - 1$$
if $C$ is edge-orientable, and
$$    \dim(\hat f_1\cap \hat f_2\cap \hat f'_3) = d - 1$$
if $C$ is non-edge-orientable,
  where $\hat f'_3$ is constructed as above in step~(v) (see also
  Figure~\ref{fig:harmonic}).
\end{proposition}

\begin{figure}[htb]
  \centering
  \begin{overpic}[width=.8\textwidth]{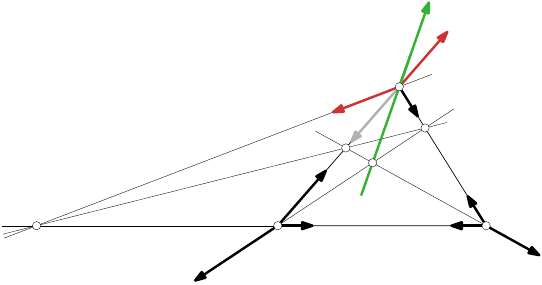}
    \put(88,8){\small$e_1$}
    \put(51,8){\small$e_2$}
    \put(70,38){\small$e_3$}
    \put(61,27){\small$g_2$}
    \put(69,19){\small$g_4$}
    \put(80,27){\small$g_3$}
    \put(5,8){\small$g_1$}
    \put(95,10){\small$\hat f_1$}
    \put(62,35){\small$\hat f_3$}
    \put(73,46){\small$\hat f'_3$}
  \end{overpic}
  \caption{
  The plane $\hat f'_3$ can be constructed using Cayley
  algebra since the two lines $f_{2, 3}, f_{3, 1}$ separate the two lines $\hat
  f_3, \hat f'_3$ harmonically.
  }
  \label{fig:harmonic}
\end{figure}
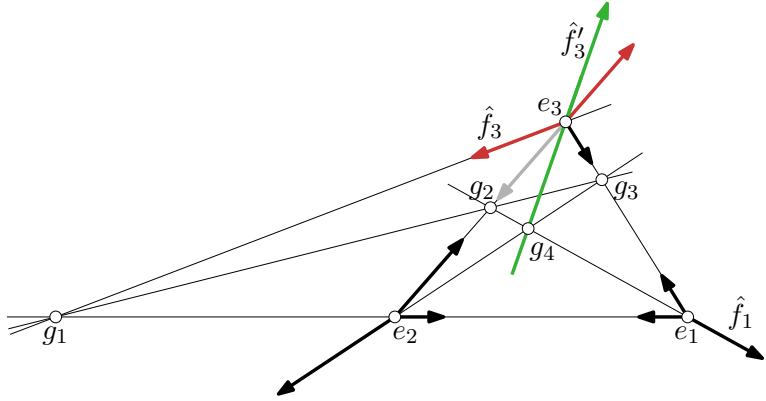

\begin{proof}
%\textbf{Step 1.}
Let us first examine the case where $d=1$ and $D=2$,
with stresses at all edges of the triangles equal to $1$.
The triangle in Figure~\ref{fig:triangle0} (left) corresponds to an
edge-oriented $d$-framework and therefore
$
\hat f_1\cap \hat f_2 \cap \hat f_3
$
is not empty.
\begin{figure}[htb]
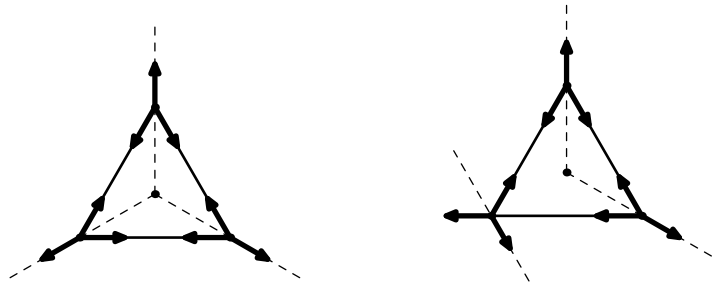

  \includegraphics{mtens-pic04.mps}
  \qquad\qquad
  \includegraphics{mtens-pic05.mps}  	
  \caption{Orientable (\emph{left}) and non-orientable (\emph{right})
  self-stressed $d$-frameworks.}
  \label{fig:triangle0}
\end{figure}
Note that this intersection is empty for the triangle in
Figure~\ref{fig:triangle0} (right), which corresponds to a non-edge-oriented
choice of normals.  The condition for non-edge-oriented tensegrities is
more complicated as we will see below.

Let us consider a self-stressed trivalent cycle  $C(E,F,\hat
F,\mathbf{n})$ of length $3$ (for a schematic sketch see
Figure~\ref{fig:triangle}). Then let us change the direction of the
normal $\mathbf{n}(e_3, f_{2,3})$.  Consequently, its orientability changes
and the old stress for this new cycle is not a self-stress.  However, by
changing the $d$-plane $\hat f_3$ to a new $d$-plane $\hat f'_3$ we can
resolve the stresses around $e_3$ again to reobtain a self-stressed
framework.
\begin{figure}[htb]
  \centering
  \begin{overpic}[width=.9\textwidth]{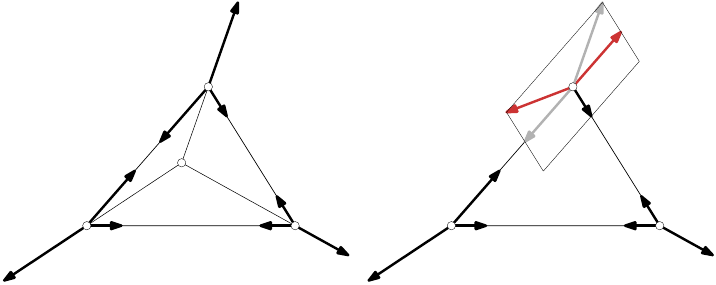}
    \put(18,19){\rotatebox{49}{\tiny$\mathbf{n}(e_3, f_{2,3})$}}
    \put(40,5){\small$e_1$}
    \put(11,5){\small$e_2$}
    \put(30,27){\small$e_3$}
    \put(24,9.5){\small$f_{1, 2}$}
    %\put(9,19){\small$f_{2, 3}$}
    \put(36,18){\small$f_{3, 1}$}
    \put(27,34){\small$\hat f_3$}
    \put(79,35){\small$\hat f_3$}
    \put(73,27){\small$\hat f'_3$}
    \put(81,26){\rotatebox{49}{\tiny$\mathbf{n}(e_3, f_{2,3})$}}
  \end{overpic}
  \caption{Self-stressed orientable and non-orientable face-cycles of
  length $3$.
  \emph{Left}: An orientable face-cycle.
  \emph{Right}: Reversing the normal $\mathbf{n}(e_3, f_{2,3})$ yields a
  non-orientable face-cycle which is still self-stressable after replacing
  $\hat f_3$ by $\hat f'_3$.
  }
  \label{fig:triangle}
\end{figure}
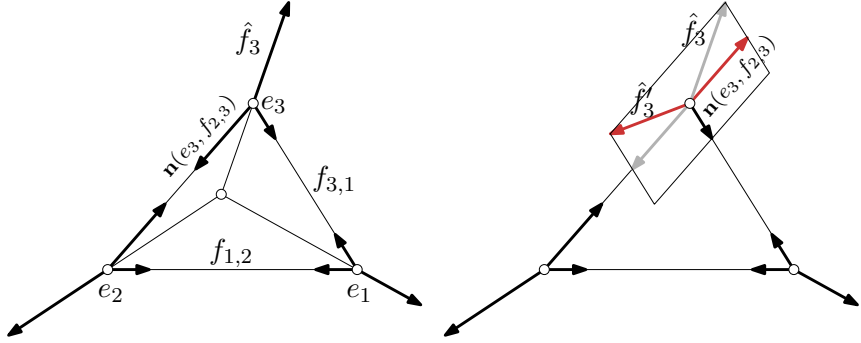

Let us consider the following two cases in the planar situation as
depicted by Figure~\ref{fig:triangle}.
The classical tensegrity (Figure~\ref{fig:triangle}), which
corresponds to
the orientable case, has the property that the lines $\hat f_1,
\hat f_2, \hat f_3$ meet in a point (see, e.g., ~\cite{Kar2018}). Now
changing the orientation of $\mathbf n(e_3, f_{2,3})$ yields $\hat f'_3$
as the new $d$-plane (see Figure~\ref{fig:triangle} right). From the
parallelogram in Figure~\ref{fig:triangle} (right) we derive the
condition for the \emph{non-orientable} case.  Standard projective geometry
implies~\cite{richter-gebert-2011}
that the two lines $f_{2, 3}, f_{3, 1}$
separate the two lines $\hat f_3, \hat f'_3$ harmonically. This property
is characterized by incidence relations of points and lines and therefore
expressible in terms of Cayley algebra.

%\textbf{Step 2.}

Next we describe how to reduce any dimension $d$ to the above one-dimensional case.
Denote by $\Pi$  the intersection
$
\Pi=f_{1,2}\cap f_{2,3}\cap f_{3,1}
$.
Observe that $\dim \Pi =d-2$.
It it well known that self-stressability is a projective invariant, see~\cite{white1983algebraic},
so we can consider the plane $\Pi$ to be at infinity.

Fix a two-dimensional plane $\pi$ orthogonal to  $e_{1}$, $e_{2}$, and $e_{3}$
(this is possible since $\Pi$ is at infinity).

Now the face-cycle $d$-framework $C$ is a Cartesian product of
$\R^{d-3}$ with the
two-dimensional tensegrity  $\mathcal{F}=C\cap \pi$ in the plane $\pi$
(see Figure~\ref{mtens-pic01}).
Stresses of $\mathcal{F}$ are in a bijection with the stresses of the initial tensegrity.
So the problem is reduced to the planar situation, i.e., to a
$1$-framework in the two-dimensional plane.

%\textbf{Step 3, edge-orientable case.
%}
%Now prove the theorem.
  Let us now consider the edge-orientable case.  According to
  Proposition~\ref{Changes} %and due to the fact that $C$ is
  %edge-orientable
   the problem has been reduced to the case of normals in
  Figure~\ref{mtens-pic01}.  The necessary and sufficient condition in the
  plane is that the three lines
$$
\hat f_1\cap \pi, \qquad
\hat f_2\cap \pi, \quad
\hat f_3\cap \pi,
$$
intersect in one point, say $a$ (see, e.g., in~\cite{Kar2018}).
Therefore $C$ is self-stressable if and only if the three planes $\hat f_1$, $\hat f_2$, and $\hat f_3$ intersect in
a common $(d{-}2)$-plane (i.e., the plane that spans $a$ and $\Pi$).
\begin{figure}[htb]
  \includegraphics{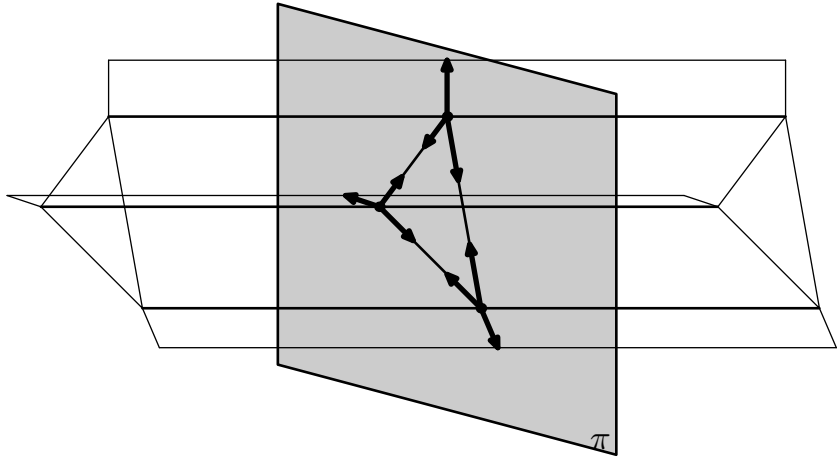}		
  \caption{A face-cycle $d$-framework $C$ and the corresponding tensegrity
  $C \cap \pi$.}
	\label{mtens-pic01}
\end{figure}

The non-edge-orientable case is reduced to the edge-orientable in the following way.
Let us make our three-cycle orientable by changing the last normal $n(e_3,f_3)$ (denote the resulting set of normals by $\mathbf n'$).
In this case, in order to preserve the property of self-stressability condition for at the edge $e_3$,
we should also change the sign of one of the coordinates for the plane $\hat f_3$.
The resulting plane is the plane $\hat f_3'$, whose Cayley algebra expression is described above (see step (v)).
Now the stressability of the original non-edge-orientable cycle is equivalent to the stressability of
an edge-oriented cycle
$$
  C\big((e_1,e_2,e_3), (f_{1,2},f_{2,3},f_{3,1}),
  (\hat f_{1}, \hat f_2, \hat f_3'),\mathbf n'\big).
$$
This concludes the proof.
\end{proof}

\subsection{$\hf$-surgeries}\label{stress-s4}
In this section we discuss $\hf$-surgeries and elementary surgery-flips on
face-paths $d$-frameworks and face-cycle $d$-frameworks which preserve self-stressability in $\R^{d+1}$.

\begin{definition}\label{HF-def}
Let $1 \le i\le n$ be positive integers ($n\ge 4$), and let
$$
C=\big((e_1,\ldots,e_n), (f_{0,1},f_{1,2},\ldots, f_{n,1}),
(\hat f_{1},\ldots, \hat f_n),\mathbf{n})\big)
$$
be a face-path (or a face-cycle if $f_{n,1}=f_{0,1}$) $d$-framework.
Denote
$$
\hat f'_i=\langle f_{i-1,i}\cap f_{i,i+1}, \hat f_i\cap \hat f_{i+1} \rangle.
$$
%Assume that
%\begin{itemize}
%\item $\dim f_{i-1,i}\cap f_{i+1,i+2}=d-1$;
%\item $\dim \hat f'_i=d$.
%\end{itemize}
We say that the \emph{$\hf_i$-surgery} of $C$ is the following
face-path (face-cycle) $d$-framework
(see Figure~\ref{figures:definition-HF})
$$
\begin{array}{rcl}
  \hf_i(C)
  =
  \big((e_1,\ldots,e_{i-1},
  &
  f_{i-1,i}\cap f_{i+1,i+2},
  &
  e_{i+2},\ldots, e_n)
  \\
  (f_{1,2},\ldots, f_{i-1,i},
  &&
  f_{i+1,i+2},\ldots, f_{n,1});
  \\
  (\hat f_1,\ldots,\hat f_{i-1},
  &
  \hat f'_{i},
  &
  \hat f_{i+2},\ldots, \hat f_n),\mathbf{n}'\big).
\end{array}
$$
\begin{figure}[htb]
\centerline{\includegraphics[width=.7\linewidth]{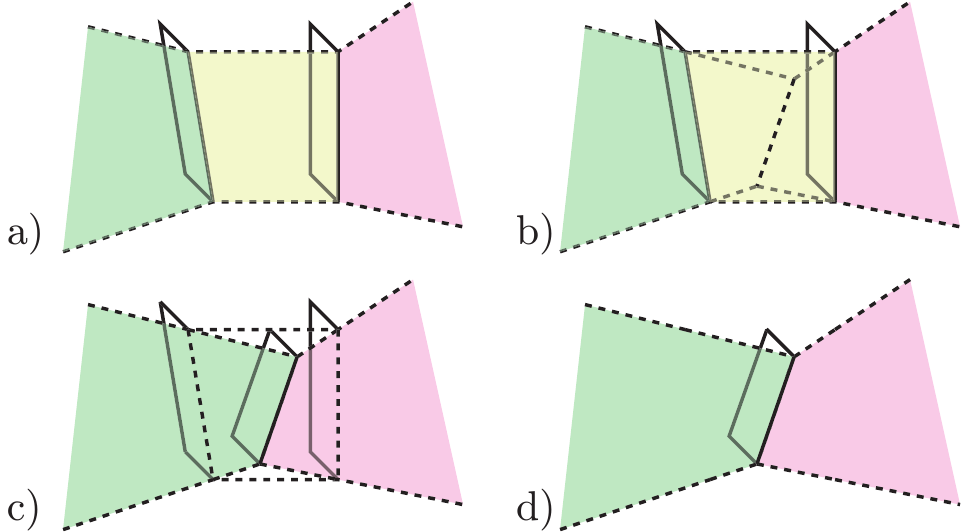}}
\caption{An $\hf_i$-surgery.}\label{figures:definition-HF}
\end{figure}
The normals  $\mathbf{n}'$ coincide with the normals of
$\mathbf{n}$ for the same  adjacent pairs.
We have three extra normals in $\mathbf{n}'$ to the new element in $E$:
$$
\begin{array}{c}
\mathbf{n}'(f_{i-1,i}\cap f_{i,i+1}, f_{i-1,i}),
\quad
\mathbf{n}'(f_{i-1,i}\cap f_{i,i+1}, f_{i,i+1}),
\quad \hbox{and}\\
\mathbf{n}'(f_{i-1,i}\cap f_{i,i+1}, \hat f_{i}).
\quad
\end{array}
$$
The first two are defined by the fact that the cycle:
$$
\Big((e_i,f_{i-1,i}\cap f_{i,i+1},e_{i+1}),
(f_{i-1,i},f_{i+1,i+2},f_{i,i+1}),
(\hat f_i,\hat f'_i,\hat f_{i+1}), \mathbf{n}'
\Big)
$$
of length 3 is edge-orientable.
The orientation of $\hat f_i$ does not play any role here
(and hence can be chosen arbitrarily).
\end{definition}

%\begin{remark}
Alternatively, in terms of Cayley algebra $\hat f'_i$ reads
$$
\hat f'_i=(f_{i-1,i}\wedge f_{i,i+1})\vee (\hat f_i\wedge \hat f_{i+1}).
$$
%\end{remark}

%\begin{remark}
In the planar case we have precisely $\hf$-surgeries on framed cycles
(i.e., face-cycle $1$-frameworks) that were used for the conditions of
planar tensegrities (for further details see~\cite{Kar2018}).
%The local picture is as follows
%$$
%\includegraphics[scale=.8]{mtens-pic02a}
%$$
%\textcolor{green}{THIS WHOLE REMARK DOES NOT MAKE SENSE TO ME.  It is the same graph.  No figure caption}
%In the $d$-dimensional case the local picture of an $\hf$-surgery is represented by taking a cone over
%the local picture of so-called $H\Phi$-surgery in two-dimensions times $\R^{d-3}$. The $H\Phi$-surgery
%was introduced in~\cite{DKS2010} for the two-dimensional case. The name for the surgery is
%motivated by the picture of the surgery in the plane: the lefthand side resembles the letter H, whereas the righthand side resembles $\Phi$.
%\end{remark}

In order to have a well-defined $\hf$-surgery, one should consider several simple conditions
on the elements of $F$ and $\hat F$.

\begin{definition}
  We say that an $\hf_i(C)$-surgery is \emph{admissible} if
  \begin{enumerate}\num
    \item\label{itttm:i} the planes $f_{i-1,i}$ and $f_{i,i+1}$ do not coincide;
    \item\label{itttm:ii} the planes $\hat f_i$ and $\hat f_{i+1}$ do not coincide;
    \item\label{itttm:iii} the planes $f_{i-1,i}\wedge f_{i,i+1}$ and
      $\hat f_i\wedge \hat f_{i+1}$ do not coincide.
  \end{enumerate}
\end{definition}

%\begin{remark}\label{forChristian}
For an admissible $\hf_i(C)$-surgery we have
\begin{eqnarray*}
\dim f_{i-1,i}\cap f_{i+1,i+2}&=&d-1;
\\\dim \hat f'_i&=&d.
\\ %\dim f_{i-1,i}\cap f_{i+1,i+2}=
   \dim \hat f_{i}\cap \hat f_{i+1}&=&d-1.
\end{eqnarray*}
where the last follows from Items~\ref{itttm:i} and~\ref{itttm:ii} above.
%we have
%$$
%\dim f_{i-1,i}\cap f_{i+1,i+2}=\dim \hat f_{i}\cap \hat f_{i+1}=d-1.
%$$
Then by Item~\ref{itttm:iii} we have $\dim \hat f'_i\ge d$.
Since the $d$-planes $f_{i-1,i}$, $f_{i+1,i+2}$, $\hat f_{i}$, and $\hat f_{i+1}$ by construction
share a $(d-2)$-plane, we have $\dim \hat f'_i\le d$.
%\end{remark}

Let us distinguish the following elementary surgery-flips.

\begin{definition}
An \emph{elementary surgery-flip} is one of the following surgeries.
\begin{itemize}
\item An admissible $\hf_i$-surgery or its inverse.

\item Removing or adding consecutive duplicates at position $i$.
Here we say that we have \emph{a duplicate at position $i$} if
$$
e_i=e_{i+1}, \quad
f_{i,i+1}=f_{i+1,i+2},
\quad \hbox{and} \quad
\hat f_i=\hat f_{i+1}.
$$
\item Removing or adding a loop of length $2$.
Here we say that we have \emph{a simple loop of length $2$ at position $i$} if
$$
e_i=e_{i+2}, \quad
f_{i,i+1}=f_{i+2,i+3},
\quad \hbox{and} \quad
\hat f_i=\hat f_{i+2}.
$$
\end{itemize}
\end{definition}

\begin{proposition}
Assuming that a surgery is admissible,
a face-cycle $d$-framework $C$ is self-stressable if and only if the face-cycle
$d$-framework $\hf_i(C)$ is  self-stressable.
\end{proposition}

\begin{proof}
Assume that $C$ has a non-zero self-stress $s$.
Let us show that $\hf_i(C)$ has a self-stress.

Consider the face-cycle $d$-framework
$$
C_i=\big((e_i,f_{i-1,i}\cap f_{i,i+1}),
(f_{i-1,i},f_{i+1,i+2},f_{i,i+1}),
(\hat f_i,\hat f'_{i},\hat f_{i+1}), \mathbf{n} \big),
$$
where
$$
\hat f'_i=\big\langle f_{i-1,i}\cap f_{i,i+1}, \hat f_i\cap \hat f_{i+1} \big\rangle.
$$
and $\mathbf{n}$ is constructed according Definition~\ref{HF-def}.
This face-cycle $d$\dash framework admits a self-stress by Proposition~\ref{three-cycle}
since three $d$-planes of $(f_i,f'_{i},f_{i+1})$
intersect in a plane of dimension $d-2$.
Now let us add $C_i$ to $C$ taking the self-stress $s_i$ which negates the stress at $e_{i,i+1}$.
Then the stresses at $\hat f_i$ for $C$ and $C_i$ negate each other;
and the stresses at $f_{i-1,i}$ for $C$ and $C_i$ coincide.
For the same reason the stresses at $f_{i+1,i+2}$ for $C$ and $C_i$ coincide.
Therefore, the constructed self-stress is in fact a non-zero self-stress on $\hf_i(C)$.

The same reasoning works for the converse statement. In fact adding the $C_i$ to $C$ provides an
isomorphism between the space of self-stresses on $C$ and the space of self-stresses on $\hf_i(C)$.
\end{proof}

\begin{remark}\label{HI-condition} It is possible to describe one $\hf$-surgery in terms of Cayley algebra.
Consider a face-cycle $d$-framework $C$ with admissible $\hf_i(C)$-surgery.
We have only one new plane $\hat f'_i$ in this case, and its Cayley expression is
$$
\hat f'_i=(f_{i-1,i}\wedge f_{i,i+1})\vee (e_i\wedge e_{i+1}).
$$
\end{remark}

%\begin{figure}[htb]
%\centerline{\includegraphics[width=.5\linewidth]{tense2019moebius03}%
%\includegraphics[width=.5\linewidth]{tense2019moebius02}}
%\caption{{ \textcolor{blue}{GAIANE DOES NOT UNDERSTAND THIS PICTURE}\color{red}{Do we write an example to this figure? Where it should be?}}}
%\end{figure}

\subsection{Stress transition and stress monodromy}\label{stress-s5}
We will now adapt to our setting the notion of ``quality transfer'' due to Rybnikov ~\cite{Rybnikov991}.

\begin{definition}
Let $\Gamma$ be a generic face-path $d$-framework with starting plane
  $f_a\in F$ and ending plane $f_{z}\in F$.
Assign some stress $s$ to the first plane. Due to genericity, it uniquely defines the stress on the second face. The stress on the second face
uniquely defines the stress on the third face, and so on. So the stress on $f_a$ uniquely defines the stress on $f_z$.
This is called the \emph{stress transition} along the face path.

If $f_a=f_z$, that is, we have a face-cycle, we arrive eventually at  some
  stress $s'$ assigned to $f_a$ again.
The ratio
%$$
%\frac{s(f_{a})}{s(f_{z})}
%$$
$s(f_{a})/s(f_{z})$
is called the \emph{stress-monodromy} along $C$. A stress monodromy of 1 is
  \emph{trivial}.
\end{definition}

It is clear that:

\begin{lemma}\label{LemmaMonodr}
\begin{enumerate}
\item A generic face-cycle is self-stressable if and only if the stress monodromy is trivial.

\item The monodromy does not depend on the choice of the first face.

\item Reversal of the direction of the cycle takes monodromy $m$ to $1/m$.

\item Monodromy behaves multiplicatively with respect to homological addition: the monodromy of the homological sum is the product of monodromies.\qed
\end{enumerate}
\end{lemma}

\subsection{Face-path equivalence}
Let us now introduce the notion of equivalent face-path $d$-frameworks.

\begin{definition}\quad
Two face-path (face-cycle) $d$-frameworks $\Gamma_1$ and $\Gamma_2$
starting from the plane $f_a$ and ending at the plane $f_z$ are \emph{equivalent}
if there exists a sequence of elementary surgery-flips taking $\Gamma_1$ to $\Gamma_2$.
\end{definition}

It turns out that equivalent face-path $d$-frameworks  have equivalent stress-transitions.

\begin{proposition}\label{Stress-translation-equivalent-paths}
The stress-transition of two equivalent face-path $d$-frameworks coincide.
\end{proposition}

\begin{proof}
It is enough to prove this statement for any elementary surgery-flip.
In case of $\hf$-surgeries we must show that the face-path $d$-frameworks
$$
C=\big((e_1,e_2),(f_{0,1},f_{1,2},f_{2,3}), (\hat f_1, \hat f_2),N\big),
$$
and
$$
\hf_1(C)=\big((e_3),(f_{0,1} ,f_{2,3}), (\hat f_3), N'\big)
$$
have the same stress-transition (see Figure~\ref{mtens-pic03}).
\begin{figure}[htb]
  \includegraphics{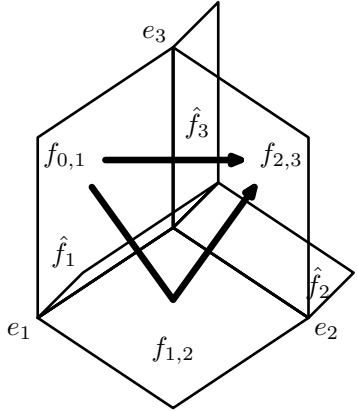}		
	\caption{An elementary flip.}
	\label{mtens-pic03}
\end{figure}
This is equivalent to the fact that the face-cycle $d$-framework
$$
\big((e_1,e_2,e_3),(f_{0,1},f_{1,2},f_{2,3}),
(\hat f_1,\hat f_2,\hat f_3), \mathbf{n}''\big)
$$
(where $\mathbf{n}''$ is as in the cycle of Definition~\ref{HF-def})  has
  a unit stress-transition (i.e.\ trivial monodromy
or, equivalently, is self-stressable).

By the construction of Definition~\ref{HF-def} we get that
 this cycle is edge-orientable, and that the intersection
$$
\hat f_1\cap \hat f_2\cap \hat  f_3\ne \emptyset.
$$
Therefore, by Proposition~\ref{three-cycle} it is self-stressable.

The cases of removing duplicates or loops of length 2 are straightforward.
\end{proof}

\subsection{Face-path $d$-frameworks in $d$-frameworks}\label{stress-s6}

In this subsection we briefly discuss face-path $d$-frameworks and face-cycle
$d$-frameworks that are parts of a larger $d$-framework.

%\begin{definition}
%  A trivalent $d$-framework is \emph{$(3,4)$-valent} if each element of
%  $V$ is incident to precisely 4 elements of $E$.
%\end{definition}
%Redfined in a later section  - no vertices here

\begin{definition}
  Let $\mathcal{F}$ be a trivalent $d$-framework.
  Then for every (cyclic) sequence of adjacent $d$-planes $\gamma$ we
  naturally associate a face-path $d$-framework (face-cycle $d$-framework)
  $\Gamma(T,\gamma)$ with
\begin{itemize}
  \item $F$ is the sequence of the planes spanned by the corresponding
    $d$-planes of $\gamma$;
  \item $E$ is the sequence of the intersections of the $d$-planes of the
    above $F$;
  \item $\hat F$ is the sequence of planes of $\mathcal{F}$ that are adjacent to the
    planes of $E$ and distinct to the faces already considered in $F$;
  \item $\mathbf{n}$ is the corresponding sequence of normals defined by
    the normals of $\mathcal{F}$.
\end{itemize}
  We say that a face-path $d$-framework (face-cycle $d$-framework)
  $\Gamma(T,\gamma)$ is \emph{induced} by $\gamma$ on $\mathcal{F}$.
\end{definition}

Induced face-path and face-cycle $d$-frameworks have a natural homotopy relation, which is defined as follows.

\begin{definition}\quad

\begin{itemize}
\item
Two induced face-path $d$-frameworks $\Gamma_1$ and $\Gamma_2$ for $\mathcal{F}$
starting from the plane $f_a$ and ending at the plane $f_z$ are \emph{face-homotopic}
if there exists a sequence of elementary surgery-flips taking $\Gamma_1$ to $\Gamma_2$ and
such that after each surgery-flip we have an induced face-path
$d$-framework for $\mathcal{F}$.

\item Two face-cycle $d$-frameworks $\Gamma_1$ and $\Gamma_2$ are \emph{face-homotopic}
if there exists a sequence of elementary surgery-flips taking $\Gamma_1$ to $\Gamma_2$
and such that after each surgery-flip we have an induced face-path
$d$-framework for $G(M)$.
\end{itemize}
\end{definition}

Finally we formulate the following important property of face\dash homotopic
face-path and face-cycle $d$-frameworks.

\begin{proposition}\label{homotopic-face-cycles-stress}
Face-homotopic face-path (face-cycle) $d$-frameworks have the same stress-transition (stress-monodromy).
\end{proposition}

\begin{proof}
The proof directly follows from Proposition~\ref{Stress-translation-equivalent-paths}.
\end{proof}

\section{Geometric Characterizations of self-stressability for Trivalent
$d$-Frameworks}\label{sectionselfstressabilitycriterium}

In this section we discuss the practical question of writing geometric
conditions for cycles. We characterize self-stressable trivalent
$d$-frameworks in terms of exact discrete multiplicative $1$-forms and in
terms of resolvable cycles. Before that we show that a trivalent
$d$-framework is self-stressable if and only if every path and every loop
is self-stressable.

\begin{definition}
A face-cycle $d$-framework is called a \emph{face-loop $d$\dash framework}
if it contains no repeating planes.
\end{definition}

Let us formulate the following general theorem.

\begin{theorem}\label{main-conditions}
  Consider a generic face-connected trivalent  $d$-framework.
  Then the following three statements are equivalent.
  \begin{enumerate}\num
    \item\label{itm:i} $\mathcal{F}$ has a non-zero self-stress $($which is in fact
      non-zero at any $d$-plane$)$.
    \item\label{itm:ii} For every two $d$-planes $f_a,f_z$ in $\mathcal{F}$ the
      stress-transition does not depend on the choice of an induced
      face-path $d$-framework on $\mathcal{F}$.
    \item\label{itm:iii} Every  induced face-loop $d$-framework on $\mathcal{F}$ is
      self-stress\-able.
  \end{enumerate}
\end{theorem}
\begin{proof}
  \ref{itm:ii} $\Leftrightarrow$ \ref{itm:i}:
  Item~\ref{itm:i} tautologically implies Item~\ref{itm:ii}.
  Let us show that Item~\ref{itm:ii} implies Item~\ref{itm:i}. Fix a
  starting face $f_a$ and put a stress $s(f_a)=1$ on it.  Expand the
  stress to all the other faces.  By  assumption this can be done
  uniquely.  Therefore, this stress is a self-stress.  (Indeed, if we do
  not have the equilibrium condition at some plane $e$, then at the planes
  incident to $e$  we have more than one possible stress-transition.)

  \ref{itm:ii} $\Rightarrow$ \ref{itm:iii}:
  Indeed any simple face-cycle $d$-framework on $\mathcal{F}$ can be considered as
  one long face-path $d$-framework with $f_a=f_z$. By condition of
  Items~\ref{itm:ii} the stress-transition equals 1, and therefore this
  face-cycle $d$-framework is self-stressable.  The last is equivalent to
  Item~\ref{itm:iii}.

  \ref{itm:iii} $\Rightarrow$ \ref{itm:ii}:
  Let us use reductio ad absurdum.  Suppose Item~\ref{itm:iii} is true
  while Item~\ref{itm:ii} is false. If Item~\ref{itm:ii} is false then
  there exist at least two face-path $d$-frameworks with the same $f_a$
  and $f_z$ where the stress-transitions fail to be the same.  Now the
  union $C$ of the first face-path $d$-framework and the inverse second is
  a induced face-cycle $d$-framework on $G(M)$ with non-unit
  stress-transition.  Let us split $C$ into consecutive loops $C_1,\ldots,
  C_k$.  At least one of them should have a non-unit translation.
  Therefore, Item~\ref{itm:iii} is false as well, a contradiction.

  For completeness of the last proof we should add the following two
  observations regarding cycles of small length.
  Firstly, the stress\dash transition remains constant at planes that
  repeat successively two or more times.  This happens due to genericity
  of $\mathcal{F}$: there are zero contributions from $\hat f_i$ in case if
  $f_{i-1,i}=f_{i,i+1}$.
  And secondly, if it happens that $f_{i-1,i}=f_{i+1,i+2}$ then we
  immediately have $e_i=e_i+1$ and therefore again the stress-transitions
  at $f_{i-1,i}$ and at $f_{i+1,i+2}$ coincide.
\end{proof}

\subsection{Ratio condition for self-stressable multidimensional trivalent
frameworks}

In this section we characterize generic trivalent $d$-frameworks $\mathcal{F}$ with
respect to their self-stressability in terms of specific products of
ratios. More precisely, we equip each $d$-framework with a so called
discrete multiplicative $1$-form which turns out to be exact if and only
if the $d$-framework is self-stressable. Let us start with the definition
of discrete multiplicative $1$-forms (see, e.g.,~\cite{bobenko+2008}).

\begin{definition}
  A real valued function $q : \vec E(G) \to \R\setminus\{0\}$
  (where $\vec E(G)$ denotes the set of oriented edges of the graph $G$) is
  called a \emph{discrete multiplicative $1$-form}, if $q(-a) = 1/q(a)$
  for every $a \in \vec E(G)$.  It is called \emph{exact} if for every
  cycle $a_1, \ldots, a_k$ of directed edges the values of the $1$-form
  multiply to $1$, i.e.,
  $$
  q(a_1) \cdot \ldots \cdot q(a_k) = 1.
  $$
\end{definition}

Now, as a next step we will equip any general trivalent $d$-framework
with a discrete multiplicative $1$-form $q$. However, we will not define
$q$ directly on the $d$-framework but on what we call its dual graph.

\begin{definition}
  The vertices of the \emph{dual graph of a $d$-framework} are the
  $d$-dimensional planes and the edges ``connect $d$-dimensional planes''
  that are sharing a $(d{-}1)$-dimensional plane.
\end{definition}

Consequently, the edges of the dual graph of $\mathcal{F}$ can be identified with
triples of successive $(d{-}1)$-planes $a_i := (e_{i - 1}, e_i, e_{i +
1})$ (where $e_i \in E$).
So let us now equip the dual graph of $\mathcal{F}$ with a discrete multiplicative
$1$-form. For an illustration see Figure~\ref{fig:facepath}.
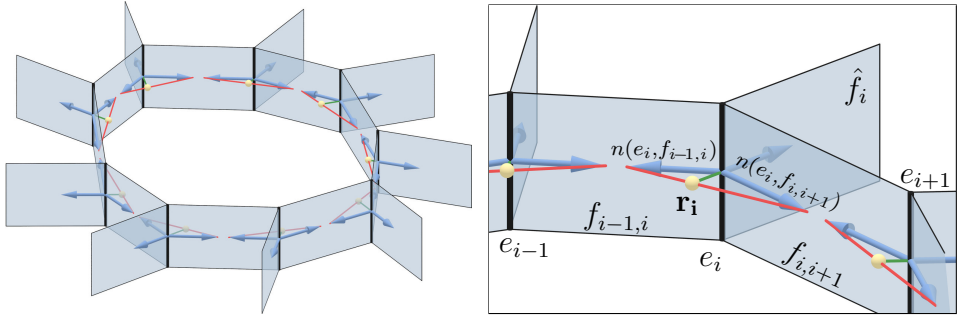
\begin{figure}[htb]

  \begin{overpic}[width=.49\textwidth]{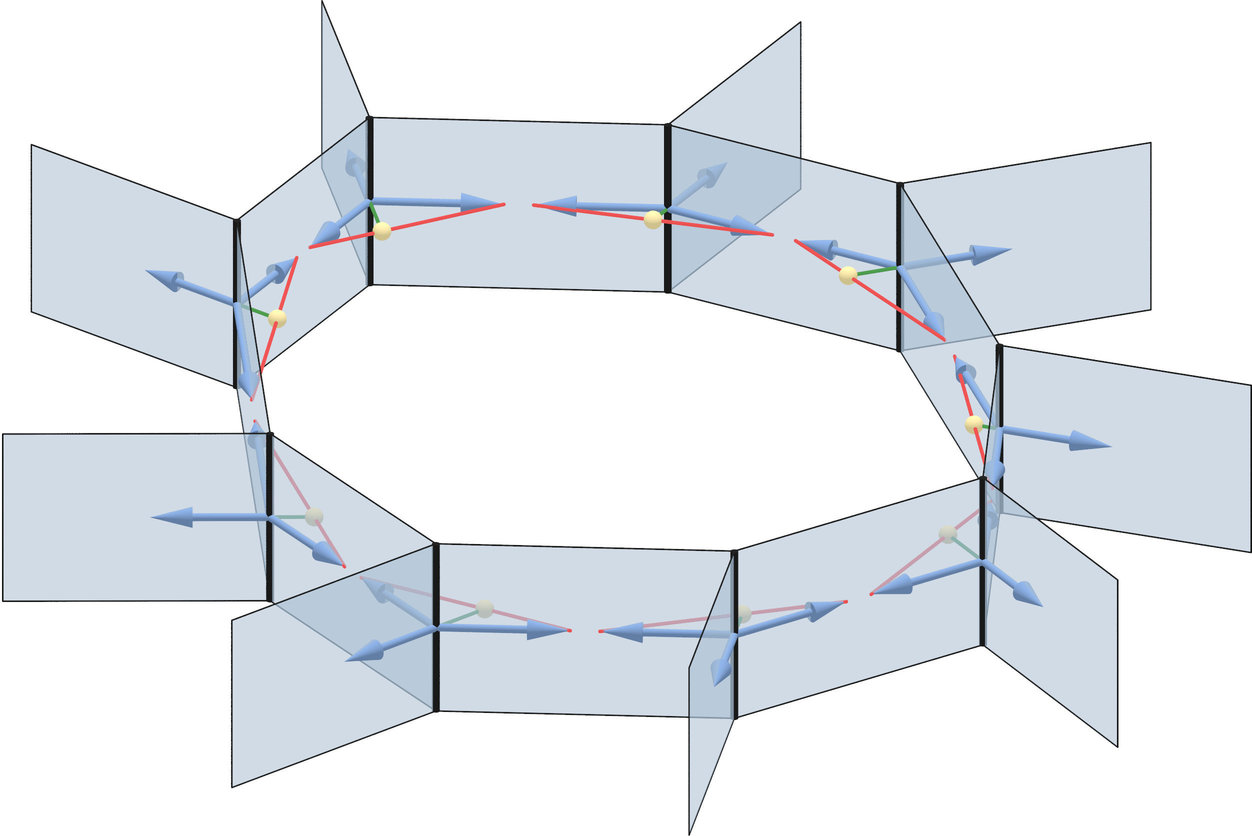}
  \end{overpic}
  \hfill
  \fboxsep0pt\fbox{%
  \begin{overpic}[width=.49\textwidth]{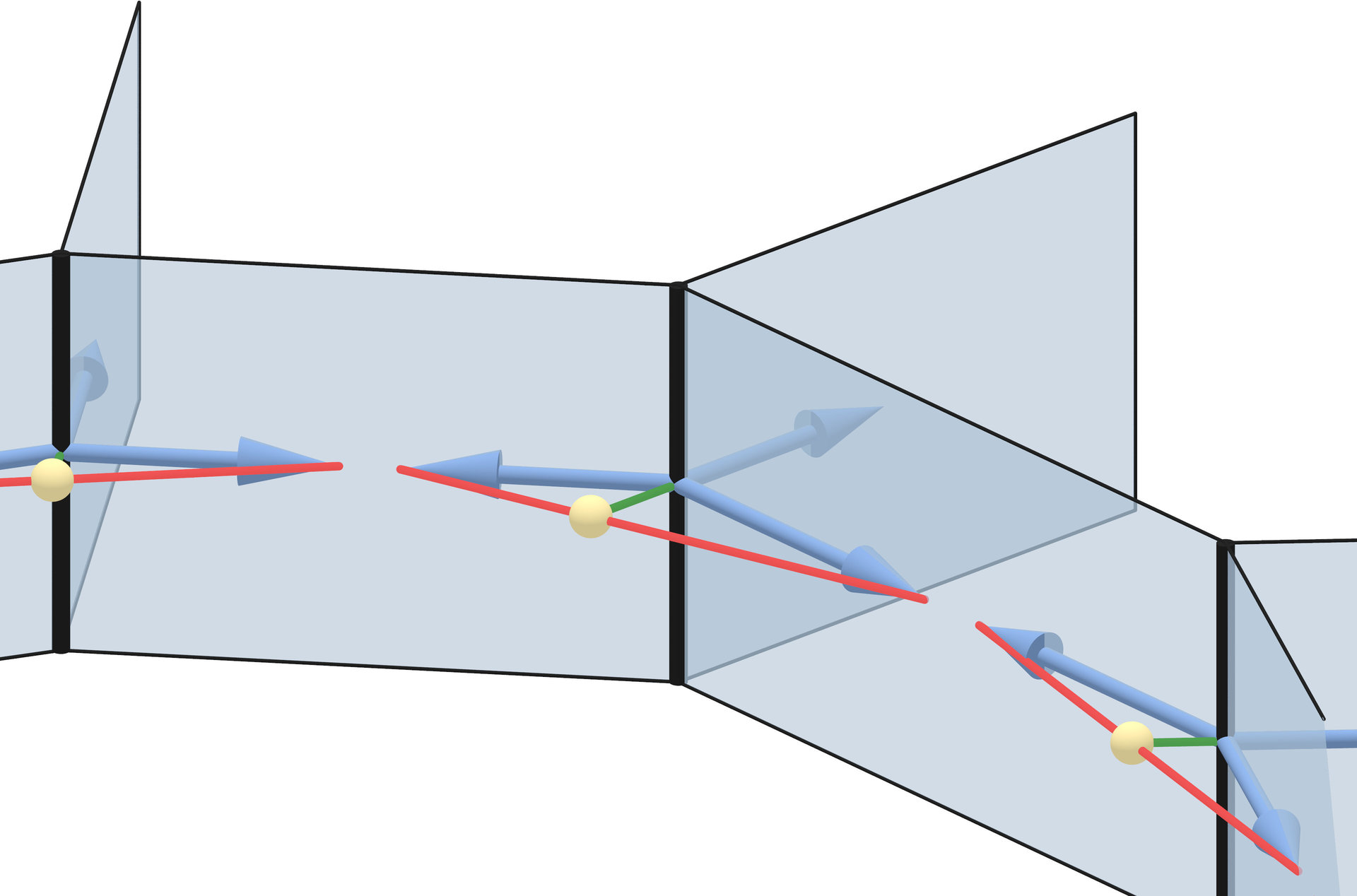}
    \put(45,10){$e_i$}
    \put(2,13){$e_{i - 1}$}
    \put(88,28){$e_{i + 1}$}
    \put(20,19){\rotatebox{-1}{$f_{i - 1, i}$}}
    \put(62,12){\rotatebox{-24}{$f_{i, i + 1}$}}
    \put(76,46){$\hat f_i$}
    \put(26,34){\rotatebox{-1}{\tiny$n(e_i,\! f_{i\! -\! 1, i})$}}
    \put(52,32){\rotatebox{-24}{\tiny$n(e_i,\! f_{i, i\!\, +\!\, 1})$}}
   \put(40,22){$\mathbf{r_i}$}
  \end{overpic}%
  }
  \caption{Illustration of some edges of the dual graph in a face loop of
  a $d$-framework.
  The values of the discrete multiplicative $1$-form $q(a_i)$ (cf.\
  Eqn.~\eqref{eq:multiform}) is the affine ratio
  $q(a_i) =
  \big(\mathbf{n}(e_i, f_{i - 1, i}) - \mathbf{r_i}\big)
  :
  \big(\mathbf{r_i} - \mathbf{n}(e_i, f_{i, i + 1})\big)$.}
  \label{fig:facepath}
\end{figure}

As our $d$-framework is trivalent the $(d{-}1)$-plane $e_i$ is contained
in three $d$-planes $f_{i - 1, i}, f_{i, i + 1}, \hat f_i$ and therefore
the corresponding normals
$\mathbf{n}(e_i, f_{i - 1, i})$,
$\mathbf{n}(e_i, f_{i, i + 1})$,
$\mathbf{n}(e_i, \hat f_i)$ are linearly dependent, i.e., lie in a
$2$-plane. This together with the fact that the $d$-framework is generic
implies that there are $\lambda_{i - 1, i}, \lambda_{i, i + 1}, \hat
\lambda \in \R\setminus\{0\}$ such that
\begin{equation}
  \label{eq:normalssum}
  \lambda_{i - 1, i} \mathbf{n}(e_i, f_{i - 1, i}) +
  \lambda_{i, i + 1} \mathbf{n}(e_i, f_{i, i + 1}) +
  \hat \lambda \mathbf{n}(e_i, \hat f) = 0.
\end{equation}

Now we are in position to define our discrete multiplicative $1$-form on
the oriented dual graph by
\begin{equation}
  \label{eq:multiform}
  q(a_i) = q(e_{i - 1}, e_i, e_{i + 1})
  :=
  \frac{\lambda_{i, i + 1}}{\lambda_{i - 1, i}},
\end{equation}
since clearly $q(-a_i) = 1/q(a_i)$ is fulfilled.
The geometric meaning of $q(a_i)$ is the following (cf.\
Figure~\ref{fig:facepath} right). Denote by $\mathbf{r_i}$ the
intersection point of the straight line with direction $\mathbf{n}(e_i,
\hat f)$ and intersect it with the line through $\mathbf{n}(e_i,
f_{i - 1, i})$ and $\mathbf{n}(e_i, f_{i, i + 1})$. A simple computation
shows
$$
\mathbf{r_i}
=
\frac{\lambda_{i - 1, i}}{\lambda_{i - 1, i} + \lambda_{i, i + 1}}
\mathbf{n}(e_i, f_{i - 1, i})
+
\frac{\lambda_{i, i + 1}}{\lambda_{i - 1, i} + \lambda_{i, i + 1}}
\mathbf{n}(e_i, f_{i, i + 1}).
$$
Thus $q(a_i)$ is the affine ratio of the three points
$\mathbf{n}(e_i, f_{i - 1, i}), \mathbf{r_i}, \mathbf{n}(e_i, f_{i, i + 1})$, i.e.,
$q(a_i) =
\big(\mathbf{n}(e_i, f_{i - 1, i}) - \mathbf{r_i}\big)
:
\big(\mathbf{r_i} - \mathbf{n}(e_i, f_{i, i + 1})\big)$.

With that definition of a discrete multiplicative $1$-form we can now
characterize self-stressable $d$-frameworks.

\begin{theorem}
  A generic trivalent $d$-framework is self-stressable if and only if
  the discrete multiplicative $1$-form defined by~\eqref{eq:multiform} is
  exact.
\end{theorem}
\begin{proof}
  Suppose the $d$-framework has a self-stress $s$. Therefore
  Equation~\eqref{self-stress-d-framework} implies
  $$
  s(f_{i - 1, i}) \mathbf{n}(e_i, f_{i - 1, i}) +
  s(f_{i, i + 1}) \mathbf{n}(e_i, f_{i, i + 1}) +
  s(\hat f) \mathbf{n}(e_i, \hat f) = 0.
  $$
  Comparison with Equation~\eqref{eq:normalssum} implies that the
  coefficients in both equations are just a multiple of each other, i.e.,
  $$
  (s(f_{i - 1, i}), s(f_{i, i + 1}), s(\hat f))
  =
  \mu (\lambda_{i - 1, i}, \lambda_{i, i + 1}, \hat \lambda).
  $$
  Consequently, the value of the discrete multiplicative $1$-form is the
  ratio of neighboring stresses:
  $$
  q(a_i) = \frac{\lambda_{i, i + 1}}{\lambda_{i - 1, i}}
  =
  \frac{s(f_{i, i + 1})}{s(f_{i - 1, i})}.
  $$
  Therefore it is easy to see that the product of values $q(a_i)$ along
  any closed loop in the dual graph multiplies to $1$:
  $$
  q(a_1) \cdot \ldots \cdot q(a_k)
  =
  \frac{s(f_{1, 2})}{s(f_{k, 1})}
  \frac{s(f_{2, 3})}{s(f_{1, 2})}
  \cdot
  \ldots
  \cdot
  \frac{s(f_{k - 1, k})}{s(f_{k - 2, k - 1})}
  \frac{s(f_{k, 1})}{s(f_{k - 1, k})}
  = 1.
  $$

  Now conversely, let us assume that the discrete multiplicative $1$-form
  $q$ is exact.
  By Theorem~\ref{main-conditions} it is sufficient to show that each loop
  of the form $e_1, \ldots, e_k$ of $(d - 1)$-planes is self-stressable.
  Choose an arbitrary stress $s(f_{1, 2}) \in \R \setminus \{0\}$ for the
  first $d$-plane. Equation~\eqref{self-stress-d-framework} and the
  $d$-framework being generic
  then uniquely determines the stresses of the two other $d$-planes
  incident to $e_1$, that is, $s(f_{k, 1})$ and $s(\hat f_1)$. Continuing,
  determining stresses this way defines all stresses along the loop
  including the last stress that we now denote by $\tilde s(f_{k, 1})$
  because it was defined before. However, the exactness of $q$ gives
%  \begin{align*}
%    1
%    &
%    \textstyle
%    =
%    q(a_1) \cdot \ldots \cdot q(a_k)
%    =
%    \frac{\lambda_{1, 2}}{\lambda_{k, 1}}
%    \frac{\lambda_{2, 3}}{\lambda_{1, 2}}
%    \cdot
%    \ldots
%    \cdot
%    \frac{\lambda_{k - 1, k}}{\lambda_{k - 2, k - 1}}
%    \frac{\lambda_{k, 1}}{\lambda_{k - 1, k}}
%    \\
%    &
%    =
%    \textstyle
%    \frac{s(f_{1, 2})}{s(f_{k, 1})}
%    \frac{s(f_{2, 3})}{s(f_{1, 2})}
%    \cdot
%    \ldots
%    \cdot
%    \frac{s(f_{k - 1, k})}{s(f_{k - 2, k - 1})}
%    \frac{\tilde s(f_{k, 1})}{s(f_{k - 1, k})}
%    =
%    \frac{\tilde s(f_{k, 1})}{s(f_{k, 1})},
%  \end{align*}
  \begin{eqnarray*}
    1
    &
        =&
    q(a_1) \cdot \ldots \cdot q(a_k)
   =
    \frac{\lambda_{1, 2}}{\lambda_{k, 1}}
    \frac{\lambda_{2, 3}}{\lambda_{1, 2}}
    \cdot
    \ldots
    \cdot
    \frac{\lambda_{k - 1, k}}{\lambda_{k - 2, k - 1}}
    \frac{\lambda_{k, 1}}{\lambda_{k - 1, k}}
   \\ &=&
    \frac{s(f_{1, 2})}{s(f_{k, 1})}
    \frac{s(f_{2, 3})}{s(f_{1, 2})}
    \cdot
    \ldots
    \cdot
    \frac{s(f_{k - 1, k})}{s(f_{k - 2, k - 1})}
    \frac{\tilde s(f_{k, 1})}{s(f_{k - 1, k})}
   \\& =&
    \frac{\tilde s(f_{k, 1})}{s(f_{k, 1})},
  \end{eqnarray*}
so  $\tilde s(f_{k, 1}) = s(f_{k, 1})$. Consequently, we can consistently
  define a non-zero stress.
\end{proof}

\subsection{Cayley algebra conditions}

Let us start with the following important definition.

\begin{definition}\label{GenericLoops}
  \quad
  \begin{itemize}
\item A face-cycle $d$-framework of length 3 is \emph{in general position} if all 6 planes in the sequences $F$ and $\hat F$
are pairwise distinct.

\item A face-cycle $d$-framework is \emph{resolvable}
if there exists a sequence of $\hf$-surgeries transforming it to a
face-cycle $d$-framework of length 3 in general position.

\item A $d$-framework is \emph{resolvable} if all its simple induced
face-cycle $d$-frameworks are resolvable.
\end{itemize}
\end{definition}

%\begin{example}
%A rather general $(3,4)$-tenting $d$-framework is resolvable
%as the conditions
%\begin{itemize}
%\item $\dim f_{i-1,i}\cap f_{i+1,i+2}=d-1$;
%\item $\dim \hat f'_i=d$.
%\end{itemize}
%are automatically fulfilled, as at every vertex $v\in V$ we have a situation
%shown on Figure~\ref{mtens-pic03}.
%\end{example}

We continue with the following definition.

\begin{definition}[\bf Cayley algebra condition for a
single face-cycle resolvable $d$-framework]\label{cycle-condition}{}
Consider a resolvable face-cycle $d$\dash framework
$$
C=(E,F,\hat F, \mathbf{n})
$$
and any sequence of $\hf$-surgeries transforming it to a
face-cycle $d$-framework
$$
C' =(E', F',\hat F', \mathbf{n}')
$$
of length 3 in general position.

\begin{itemize}
\item Let us write all elements of $F'$ and $\hat F'$ in $C'$ as Cayley algebra expressions
of the elements of $F$ and $\hat F$ in $C$. The resulting expressions are compositions of
expressions of Remark~\ref{HI-condition}.

\item Finally, we use the dimension condition of Proposition~\ref{three-cycle} for $C'$
to determine if $C'$ is stressable or not.
\end{itemize}
The composition of the above two items gives an existence condition for nonzero self-stresses on $C$.
We call this condition a {\it Cayley algebra geometric condition} for $C$
to admit a non-zero self-stress.
\end{definition}

%\begin{remark}
Note that one can write distinct Cayley algebra geometric condition for $C$ using different sequences of
$\hf$-surgeries transforming $C$ to a face-cycle $d$-framework of length 3 in general position.
%\end{remark}

As we have already mentioned in the above definition
the Cayley algebra geometric conditions detect self-stressability on cycles.
So we can write self-stressability conditions for a general trivalent resolvable $d$-framework.
Namely we have the following theorem.

\begin{theorem}\label{ccc1}
Let $\mathcal{F}$ be a trivalent resolvable $d$-framework and let further $C_1,\ldots, C_n$ be all pairwise non-face-homotopic face-loop $d$-frame\-works on $\mathcal{F}$.
Then $\mathcal{F}$ has a self-stress if and only if it fulfills Cayley algebra geometric conditions for
cycles $C_1,\ldots, C_n$ as in Definition~\ref{cycle-condition}.
\end{theorem}

\begin{proof}
First of all the self-stressability $C_1,\ldots, C_n$ is equivalent to self-stressability of all face-loops on $\mathcal{F}$.
It follows directly from definition of face-homotopic paths.
Hence by Theorem~\ref{main-conditions}
the self-stressability of all $C_1,\ldots, C_n$ is equivalent
to self-stressability of $\mathcal{F}$ itself.

Finally the geometric conditions for face-loop $d$-frame\-works $C_i$ ($i=1,\ldots, n$)
are described in Definition~\ref{cycle-condition}.
\end{proof}

\begin{corollary}\label{corK5} 
Each realization of $K_5$ in $\mathbb{R}^3$  is self-stressable.
Here we mean a realization of a $2$-framework associated with $K_5$, 
in the spirit of Example~\ref{K5stresssEx}. 
Namely, the edges are all the edges of $K_5$, and faces are all the associated triangles.
\end{corollary}
Indeed, each face-loop in $K_5$ is face-homotopic to a face-loop of length three. For them, the stressability condition is automatic.
Another argument of stressability of $K_5$ will appear later in Remark~\ref{remark-last} as a consequence of Theorem~\ref{ThmRecipStressLift}.

%%%%%%%%%%%%%%%%%%%%%%%%%%%%%%%%%%%%%%%%%%%%%%%5

\section{R-frameworks and their self-stressability. Examples.}\label{sectionstressabilityII}

In this section we work in the settings of Rybnikov's papers \cite{Rybnikov991,Rybnikov99}.

\subsection{R-frameworks}

Informally, R-frameworks are PL (piecewise linear) realizations of
CW-complexes in $\mathbb{R}^{d+1}$. To make this precise,  let us start with a  reminder
about CW complexes.

A finite CW-complex is  constructed inductively by defining its skeleta
(for details see, e.g.,~\cite{Hatcher}). The zero skeleton  $\sk_0$ is  a
finite set of points called \emph{vertices}. Once the $(k{-}1)$-skeleton
$\sk_{k-1}$ is constructed, a finite collection of closed $k$-balls $B_i$
(called \emph{cells}) is attached by some continuous mappings
$\phi_i:\partial B_i \rightarrow \sk_{k-1}$. The images of $B_i$ in the complex are called \emph{closed cells}.

\begin{definition}
  A \emph{regular} CW-complex is a CW-complex such that
  \begin{enumerate}\num
    \item For each $k$-cell $B_i$, the mapping $\phi_i $ is a
      homeomorphism between $\partial B_i$ and a subcomplex of the skeleton
      $\sk_{k-1}$.
    \item The intersection of two closed cells is either empty or some single
      closed cell of this CW-complex.
  \end{enumerate}
\end{definition}

Let $M$ be a regular finite CW-complex with no cells of dimension greater than $d$.

Its faces of dimension $d$ will be called \emph{$d$-faces}.
The $(d-1)$-faces are called the \emph{$d$-edges}.
The $(d-2)$-faces are called the \emph{$d$-vertices}.

\begin{example}
  Let $\overline{M}$  be a regular finite  CW-complex whose
  support\footnote{The support of a CW-complex is the topological space
  represented by the complex. That is, one forgets the combinatorics and
  leaves the topology only.}  $|\overline{M}|$ is a connected
  $(d{+}1)$-manifold, either closed or with boundary.  Let $M$ be its
  $d$-skeleton.
\end{example}

In this setting we also have cells of $\overline{M}$ of dimension $d+1$.
These  will be called \emph{chambers}.

\begin{definition}
  Assume that a mapping $p:\mathrm{Vert}(M)\rightarrow \mathbb{R}^{d+1}$ is such
  that the image of the vertex set of each $k$-cell spans some affine
  $k$-plane.

  We say that $p$ \emph{realizes} $M$ in $\mathbb{R}^{d+1}$; we also say
  that the pair $(M,p)$ is  a \emph{realization} of $M$, or a
  \emph{Rybnikov-framework}, or \emph{R-framework}, for short.
\end{definition}

Notation: given a cell $f\in M$, we abbreviate the image of the vertex
set $p(\mathrm{Vert}(F))$ as $p(f)$
and denote by $\langle p(f) \rangle$ its affine span.

\begin{definition}
  An R-framework is \emph{generic} if, whenever two $d$-faces $f_1$ and $f_2$
  share a $d$-edge, then $\langle p(f_1)\rangle \neq \langle p(f_2)\rangle$.

  From now on we assume that all R-frameworks we deal with are generic.

  As before, we say that an R-framework is   \emph{trivalent} if each $d$-edge is incident to exactly three
  $d$-faces.

   An R-framework is  $(3,4)$\emph{-valent} if each $d$-vertex is incident to exactly
  four $d$-edges (and therefore, to six $d$-faces).
\end{definition}

{\textbf{Informal remark:}  by construction, vertices are mapped to points. One may also imagine  that  the $1$-cells of the complex are mapped to line segments. Therefore $2$-cells are mapped to some closed planar broken lines (polygons). Here self-intersections may occur. As the dimension of  faces grows, the more complicated the associated geometrical object is.

 %In 3D we encounter polyhedra, whose faces might have self-intersection  (already this needs a definition).  Therefore we prefer not to think of these objects and work with $p$ with affine hulls of the images (that is, well-understood lines, planes,etc).

However, there exist nice examples with  convex polyhedra as
images of the faces.  In particular, if a face is a (combinatorial)
simplex, one may think that its image is  a simplex lying in
$\mathbb{R}^d$.}

\begin{example}\label{ExSchlegel}
  \begin{enumerate}\num
    \item \emph{The Schlegel diagram} \cite{Ziegler} of a convex
      $(d+2)$-polytope $K$ is a realisation   of the boundary complex of $K$.

    \item More generally, the projection of  a $(d{+}2)$-dimensional
       polyhedral body $K$ (that is, of a body with piecewise linear
      boundary) to  $\mathbb{R}^{d+1}$ yields a realization  $(M,p)$ where
      $|M|$ is homeomorphic to $\partial K$.
  \end{enumerate}
\end{example}

\subsection{Self-stresses and liftings}
%A lift is a representation of a R-framework as a  as a projection of a $(d+1)$-dimensional PL terrain.
Now we turn to a particular notion of stresses, which
 is borrowed from Rybnikov's paper \cite{Rybnikov991}  and  represents a special case of Definition~\ref{dframeworkdefinition}.
 The principal difference  is that in Rybnikov's setting the choice of normal vectors $\mathbf{n}$
 is dictated by the R-framework.

As Examples~\ref{K5stresssEx} and \ref{ExSchlegel} show, in certain cases, a realization
$(M,p)$ represents all the faces as convex polytopes. Let us call such a
realization \emph{non-crossing}. Otherwise, we say that the realization
is \emph{self-crossing}.

Let us start by introducing stresses for the \textbf{non-crossing version}:
\begin{definition}[cf.\ \cite{Rybnikov99}]
  Assume that a non-crossing realization $(M,p)$ is fixed. Let us assign
  to each pair $(f,e)$ where $e$ is a $d$-edge contained in a $d$-face
  $f$ a unit normal $\mathbf{n}(e,f)$ to $p(e)$ pointing inside the convex
  polytope $p(f)$.

  A real-valued function $s$ defined on the set of $d$-faces is called a
  \emph{self-stress} if at each $d$-edge $e$ of the complex,
  $$
  \sum_{f\supset e}s(f)\mathbf{n}(e,f)=0.\ \ \ (*)
  $$
\end{definition}

%{\textbf{Remark for internal use:}  We use here (a very slightly modified)  Rybnikov's narration  borrowed from
%\newline www.collectionscanada.gc.ca/obj/s4/f2/dsk2/ftp03/NQ45268  page 21. }

To relax the non-crossing condition, let us make some preparation following ~\cite{Rybnikov99}. The informal idea is to triangulate the faces of the complex, since the
representation of a simplex is never self-crossing.

Pick a (combinatorial) orientation of each of the cells of $M$, and
 a (combinatorial) triangulation of $M$ without adding new vertices. So each $d$-face now is replaced by a collection of (combinatorial) simplices. The realization $(M,p)$ yields a realization $(\overline{M},p)$ of the new CW-complex.

\begin{definition}[cf.\ \cite{Rybnikov99}]\label{DefStr}
  Assume that a generic realization $(M,p)$ is fixed. Choose a
  triangulation $(\overline{M},p)$ as is described above.

  For a $d$-edge $e\in \overline{M}$ and a $d$-face $f$ containing $g$,
  choose $\mathbf{n}(e,f)$ to be the unit normal to the oriented cell $f$
  at its simplicial face $g$ whose orientation is induced by the
  orientation of $f$.

  A real-valued function $s$ on the set of $d$-cells of $M$ is called a
  \emph{self-stress} if for every  $d$-edge $e$ of $\overline{M}$, the condition $(*)$ is fulfilled.
\end{definition}

This definition is proven to be independent on the choice of the combinatorial triangulation and also on the choice of the orientations of the faces.

The notion of stressed realizations has the following physical meaning. One imagines that the $d$-faces are realized by planar soap film.
The faces are made of different types of soap, that is, with different physical property. Each of the faces creates a tension, which should
be equilibrium at the $d$-edges. The tension is always orthogonal to the boundary of a face and lies in the affine hull of the face.
A self-intersecting face produces both compression and tension as is depicted in Figure~\ref{tense2019nonconvex01Fig}.
\begin{figure}[htb]
  \centering
  \includegraphics{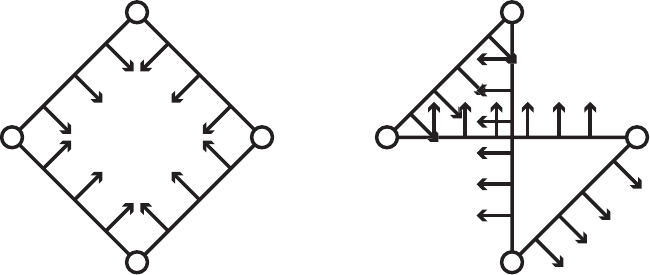}
  \caption{Tensions for a convex quadrilateral (left) and a self intersecting quadrilateral (right).}
  \label{tense2019nonconvex01Fig}
\end{figure}

We say that a R-framework $(M,p)$ is \emph{self-stressable} whenever there exists a non-zero stress.

\begin{proposition}The following two statements hold.
  \begin{enumerate}\num
    \item Each R-framework yields a $d$-framework $(E,I,F,\mathbf{n})$ which
      agrees with the Definition~\ref{dframeworkdefinition}. The
      incidences are dictated by combinatorics of $M$. Its
      self-stressability agrees with the Definition~\ref{tensegritydefinition}
      \qed
    \item For $d=1$, an R-framework  is a planar realization of
      some graph. Its self-stressability agrees with the classical notion
      of self-stresses of graphs in the plane.
      \qed
  \end{enumerate}
\end{proposition}

Assume now that $M$ is the $d$-skeleton of some $(d{+}1)$-dimensional
manifold $\overline{M}$, that is, the chambers are well-defined.

\begin{definition}
  A \emph{lift} of $(\overline{M}, p)$ is an assignment of a linear function
  $h_C:\mathbb{R}^d\rightarrow \mathbb{R}$ to each chamber $C$. By
  definition, a lift satisfies the following: whenever two chambers $C$
  and $C'$ share a $d$-face $f$, the restrictions of $h_C$ and $h_{C'}$
  on the affine span $\langle p(f)\rangle$ coincide.

  A lift is non-trivial if (at least some of) the functions $h_C$ are
  different for different chambers.
\end{definition}

Let us fix some  chamber $C$. Lifts that are identically zero on $C$  form a linear
space $\mbox{Lift}(\overline{M},p)$.

\iffalse
Remark. Assume that $C$ and $C'$ are two chambers that share a face. Assume that there is a lift of $(M,p)$
such that the linear functions $h_C$ and $h_{C'}$ are different. Denote by $e_C$ and $e_{C'}$ the graphs of $h_C$ and $h_{C'}$ respectively. These are two hyperplanes in $\mathbb{R}^{d+1}$. Then the orthogonal projection of $e_C\cap e_{FC'}$ to $\mathbb{R}^d$
equals $<p(C\cap C')>$. Actually, this \emph{projection condition} rephrases the definition.
\fi

\begin{theorem}[cf.\ \cite{Rybnikov99}]\label{ThmRecipStressLift}
  Let $M$ be the $d$-skeleton of some $(d{+}1)$-dimensional manifold
  $\overline{M}$.
  \begin{enumerate}\num
    \item If the first homology group of $\overline{M}$ vanishes, that is,
      $$H_1(M,\mathbb{Z}_2)=0,$$ then the linear spaces
      $\mathrm{Lift}(\overline{M},p)$ and the space of self-stresses
      $\mathrm{Stress}(M,p)$ are canonically isomorphic.
    \item Liftability of $(\overline{M},p)$ implies  self-stressability of $(M,p)$.
      \qed
  \end{enumerate}
\end{theorem}

\begin{remark}\label{remark-last}
The theorem gives another proof of Corollary~\ref{corK5} 
Indeed, each realization of $K_5$ can be viewed 
as a projection of a four-dimensional simplex. 
In other words, it is liftable, and hence stressable.
\end{remark}

Each $d$-vertex $v$  of a R-framework yields in a natural way
a spherical framework via the following algorithm:
 \begin{enumerate}
 \item We may assume that each face is a simplex, otherwise triangulate the faces.
   \item Take an affine
$h$ which is orthogonal to the affine span $\langle p(v)\rangle$. Clearly, we have $dim h=3$.
   \item Take a small sphere $S^2$ lying in the {plane} $e$
and centered at the intersection point $O= h\cap \langle p(v)\rangle$.
   \item For each $d$-face $f$ incident to $v$, take the projection $Pr_h (f)$ to the plane $h$ and the intersection $Pr_h (f)\cup S^2$. Since $f$ is a simplex,
   the intersection is a geodesic arch.
 \end{enumerate}
This yields a framework $\mathcal{S}_v$ placed in the sphere  $S^2$.
Self-stressability of spherical graphs is well understood since it reduces
to self-stressability of planar graphs (see
\cite{crapo+1993,panina-2009,streinu+2005}).
A face loop is called \emph{local with respect a $d$-vertex $v$} if all the
$d$-faces and $d$-edges participating in the path are incident to $v$.

\begin{lemma}\label{LocalMon}
  The two statements are equivalent:
  \begin{enumerate}\num
    \item The stress monodromy of each local (writh respect to some  $d$-vertex $v$)
      face loop is trivial.
    \item The spherical framework $\mathcal{S}_v$ is self-stressable.
  \end{enumerate}
\end{lemma}
\begin{proof}
  Triviality of any local stress monodromy implies that stresses can be
  assigned to the faces incident to $v$ in such a way that locally the
  equilibrium condition holds.

  The same stress assignment gives a self-stress of
  $\mathcal{S}_v$, and vice versa.
\end{proof}

\begin{example}
  If a $d$-vertex $v$ has exactly four incident $d$-edges, then
  $\mathcal{S}_v$ is stressable. Indeed, in this case $\mathcal{S}_v$ is a
  $K_4$ placed on the sphere, which is always stressable.
\end{example}

\begin{theorem}\label{ThmRframeworkStressability}
  Assume that $\mathcal{R}$ is a trivalent R-framework.
  \begin{enumerate}\num
    \item $\mathcal{R}$ is self-stressable iff for each face loop the
      stress monodromy is trivial.
    \item $\mathcal{R}$ is self-stressable iff the two conditions hold:
      \begin{enumerate}
        \item For each vertex $v$, the induced spherical framework
          $\mathcal{S}_v$ is self-stressable.
        \item For some generators $g_1,\ldots,g_k$ of the first homolog
          group $H_1(\mathcal{R})$, and some collection of representatives
          $\gamma_1,\ldots,\gamma_k$ that are face-cycles, all the stress
          monodromies are trivial. (One representative for one generator).
      \end{enumerate}
    \item In particular, if $\mathcal{R}$ is one-connected, its
      self-stressability is equivalent to self-stressability of
      $\mathcal{S}_v$ for all the $d$-vertices.
  \end{enumerate}
\end{theorem}
\begin{proof}
  (i) Take any face, assign to it any stress, and extend it to other
  faces. Triviality of the monodromy guarantees that no contradiction will
  arise.

  (ii) Keeping in mind Lemma~\ref{LemmaMonodr}, observe that any face
  loop is a linear combination of $\gamma_1,\ldots,\gamma_k$ and some
  local face loops.
  By Lemma~\ref{LocalMon}, the monodromies of local loops are trivial. It remains to apply (i).
 Now (iii) follows.
\end{proof}

\subsection{Some examples}

Let us start by two elementary examples.

(1)
Take the Schlegel diagram of a $3$-dimensional cube (that is the
projection of the edges of a cube). It is a trivalent graph in the plane.
It is self-stressable since it is liftable by construction.

However, one easily can redraw it keeping the combinatorics in such a way
that the realization is no longer self-stressable. For this, it is sufficient to  generically perturbe the positions of the vertices.

(2)
Now let us work out an analogous example in $\mathbb{R}^3$. Take the
Schlegel diagram of a $4$-dimensional cube (that is, the projections of all
the $2$-faces). It is a $(3,4)$-valent  R-framework. It is self-stressable since it
is liftable by construction.
But, unlike (1), by Theorem \ref{ThmRframeworkStressability},
any other its  realization is self-stressable.

The above examples  suggest   general questions:

\emph{ Are all trivalent R-frameworks self-stressable? Are all $(3,4)$-valent R-frameworks self-stressable?}

The answer is negative, which is demonstrated through the below example,
which is interesting for its own sake.

\begin{example}
  Take a prism $P$ in $\mathbb{R}^3$, a pyramid over $P$, and the
  projection of the $2$-skeleton of the pyramid back to $\mathbb{R}^3$. We
  obtain a R-framework $\mathcal{R}$ which is defined with some freedom:
  firstly, one may alter the position of the vertex of the pyramid, and
  secondly, one may apply a projective transform to $P$.
\end{example}

\begin{lemma}
  The R-framework $\mathcal{R}$ is self-stressable, and the space of stresses
  is one-dimensional.
\end{lemma}
\begin{proof}
  The R-framework $\mathcal{R}$ is a projection of the $2$-skeleton of
  $4$-dimensional tetrahedron. In other words, $\mathcal{R}$ is liftable,
  and therefore, self-stressable. Since it is trivalent,
  the space of stresses is at most one-dimensional.
\end{proof}

The prism $P$ has two disjoint triangular faces. They are also faces of
$\mathcal{R}$; let us call them \emph{green}.
The edges of these faces are also called \emph{green}. The faces of
$\mathcal{R}$ that are not green are called \emph{white}. The edges that
are not green are called \emph{white}.
Fix also one of the white faces, let us call it the \emph{test face} for
$\mathcal{R}$.

It is easy to check that any two of white faces of $\mathcal{R}$ are
connected by a face path which uses white edges only.

\noindent
\textbf{Main construction, first step:}

Take two copies of $\mathcal{R}$, say, $\mathcal{R}_1$ and $\mathcal{R}_2$
such that one of the green faces of $\mathcal{R}_1$ coincides with a green
face of  for $\mathcal{R}_2$. Patch $\mathcal{R}_1$ and $\mathcal{R}_2$
along these faces, and eliminate these green faces. Define the result by
$\mathcal{R}_{12}$.

\begin{lemma}
  $\mathcal{R}_{12}$ is self-stressable, and the space of the  self-stresses
  has dimension $1$.
\end{lemma}
\begin{proof}
  { Self-stressability:}
  Take a self-stress $s_1$ of $\mathcal{R}_1$ and a stress $s_2$ of
  $\mathcal{R}_2$ such that $s_1-s_2$ on the patched green faces
  vanishes.
  Then $s_1-s_2$ represents a non-trivial stress of $\mathcal{R}_{12}$.

  { Dimension one:} Consider a
  stress $s$ on $\mathcal{R}_{12}$.
  Take a stress $s_1$ of $\mathcal{R}_1$ which agrees with $s$ on the test
  face of $F_1$ and a stress $s_2$ of $\mathcal{R}_2$ which agrees with
  $s$ on the test face of $\mathcal{R}_2$.
  Take $s-s_1-s_2$. It is a stress on $\mathcal{R}_{12}$ plus the green
  face which is zero everywhere on $\mathcal{R}_{12}$, except, may be,
  the green face, which possible only if $s-s_1-s_2$ vanishes everywhere.
  Therefore, each stress of $\mathcal{R}_{12}$ is a linear combination of
  two stresses of $\mathcal{R}_1$ and $\mathcal{R}_2$ that cancel each
  other on the green face.
\end{proof}

\textbf{Main construction, second step:}
We proceed in the same manner: we take one more copy of $\mathcal{R}$,
which is called $\mathcal{R}_3$ and patch it to $\mathcal{R}_{12}$ along
green faces, and eliminate the green face which was used for the patch. We
get $\mathcal{R}_{123}$. Analogously, we have:

\begin{lemma}
  $R_{123}$ is self-stressable, and the space of the stresses has dimension
  $1$.\qed
\end{lemma}

\textbf{Main construction, next steps:}
Now we have a chain of three copies of $\mathcal{R}$ patched together.
Only two green faces survive.  By adjusting the shapes of the components,
we may assume that these green faces coincide. Patch the last two
green faces and remove them.
After that,  we get a R-framework $\widetilde{\mathcal{R}}$. Generically we
have:

\begin{lemma}
  $\widetilde{\mathcal{S}}$ is not self-stressable, but it is locally
  self-stressable, that is  $\mathcal{S}_v$ is self-stressable for
  each vertex $v$.
\end{lemma}
\begin{proof}
 {Non-self-stressability:} Before the last step, the space of
  stresses was one-dimensional. After the last step (=after patching two
  last green faces), the dimension can only drop.
  Let us patch back the two last green faces and get a framework
  $\widetilde{\mathcal{R}}'$.

  Assume $\overline{\mathcal{R}}$ is self-stressable. This means that
  $\widetilde{\mathcal{R}}'$ has a stress which sums up to zero on the
  last two green faces.
  By the above lemmata, the value of the stress on the test face of
  $\mathcal{R}_1$ uniquely defines the stress on the first green face and
  the stress on the test face of $\mathcal{R}_2$ (and all the faces of
  $\mathcal{R}_1$ and $\mathcal{R}_2$), this uniquely defines the stress
  on the test face of $\mathcal{R}_3$, and so on. We conclude that the stress on the
  second green face is also uniquely defined, and generically, the green
  stresses do not cancel each other.

{Local self-stressability} follows from the fact that putting
  back any of the green faces creates a self-stressable R-framework.
\end{proof}

Let us observe that $\widetilde{\mathcal{R}}$ is trivalent, but not $(3,4)$-valent.
Each of the green edges has four incident faces, and there are vertices of
valency higher that $4$.
However one can prove that a number of local surgeries turns it to a
$(3,4)$-valent R-framework.

Eventually we arrive at a $(3,4)$ R-framework which is locally
self-stressable, but not globally self-stressable.

\section*{Acknowledgement.}

The collaborative research on this article was initiated during
``Research-In-Groups'' programs of ICMS Edinburgh, UK. The authors are
grateful to ICMS for hospitality and excellent working conditions.

Christian M{\"u}ller gratefully acknowledges the support of the
Austrian Science Fund (FWF) through project P~29981.

\bibliographystyle{plain}
\bibliography{tensegrity}

\end{document}